\documentclass[11pt]{article}
\usepackage{amsmath,amssymb,amsthm}

\usepackage{enumerate,subfigure,multirow,epsfig,psfrag,hhline,color,srcltx,subeqnarray}
\usepackage{tabularx}
\usepackage{tabulary}
\usepackage{xcolor}
\usepackage{hyperref}
\DeclareGraphicsExtensions{.pdf,.png,.jpg}
\usepackage[sort]{cite}

\usepackage{tcolorbox}
\hypersetup{breaklinks=true,colorlinks=true,linkcolor=blue,citecolor=blue}
\usepackage{appendix}
\usepackage{graphicx}
\usepackage{graphics}

\usepackage{algorithm}
\usepackage{algpseudocode}

\usepackage[noabbrev,capitalise,nameinlink]{cleveref}
\usepackage[round]{natbib}

\newtheorem{assumption}{Assumption}
\newtheorem{remark}{Remark}
\newtheorem{lemma}{Lemma}
\newtheorem{example}{Example}

\newtheorem{proposition}{Proposition}
\newtheorem{theorem}{Theorem}

\newtheorem{definition}{Definition}

\DeclareMathOperator*{\argmax}{arg\,max}

\allowdisplaybreaks

\title{A Unified Switching System Perspective and O.D.E. Analysis of Q-Learning Algorithms}

\author{Donghwan Lee
\thanks{D. Lee is with Coordinated Science Laboratory (CSL),
University of Illinois at Urbana-Champaign, IL 61801, USA {\tt\small
donghwan@illinois.edu}.}$\,\,$ and Niao He% <-this % stops a space
\thanks{N. He is with the Department of Industrial and Enterprise Systems Engineering,
University of Illinois at Urbana-Champaign, IL 61801, USA {\tt\small
niaohe@illinois.edu}.}}

\begin{document}

%\maketitle \thispagestyle{empty} \pagestyle{empty}
\maketitle

\begin{abstract}
 In this paper, we introduce a unified framework for analyzing a large family of Q-learning algorithms, based on switching system perspectives and ODE-based stochastic approximation. We show that the nonlinear ODE models associated with these Q-learning algorithms can be formulated as switched linear systems, and analyze their asymptotic stability by leveraging existing switching system theories. Our approach provides the first O.D.E. analysis of the asymptotic convergence of various Q-learning algorithms, including asynchronous Q-learning and averaging Q-learning. We also extend the approach to analyze Q-learning with linear function approximation and derive a new sufficient condition for its convergence.

\end{abstract}

%% REQUIRED
%\begin{keywords}
%  Q-learning, switched linear system, stochastic approximation
%\end{keywords}
%
%% REQUIRED
%\begin{AMS}
%  68Q25, 68R10, 68U05
%\end{AMS}

%%%%%%%%%%%%%%%%%%%%%%%%%%%%%%%%%%%%%%%%%%%%%%%%%%%%%%%%%%%%%%%%%%%%%%%%%%%%%%%%
\section{Introduction}
 Reinforcement learning (RL) addresses the optimal control problem for unknown systems through experiences~\citep{sutton1998reinforcement}.
Q-learning, originally introduced by Watkin~\citep{watkins1992q}, is one of the most popular and fundamental reinforcement learning  algorithms for unknown systems described by Markov decision processes. The convergence of Q-learning has been extensively studied in the literature and proven via several different approaches, including the original proof~\citep{watkins1992q}, the stochastic approximation and contraction mapping-based approach~\citep{jaakkola1994convergence, tsitsiklis1994asynchronous,bertsekas1996neuro,even2003learning,szepesvari1998asymptotic,azar2011speedy,zou2019finite}, and the ODE (ordinary differential equation) approach~\citep{borkar2000ode}.

The ODE approach analyzes the convergence of general stochastic recursions by examining stability of the associated ODE model~\citep{bhatnagar2012stochastic,kushner2003stochastic,borkar2000ode} and has been used as a convenient analysis tool to prove convergence of many RL algorithms, especially the temporal difference (TD) learning algorithm~\citep{sutton1988learning} and its variants~\citep{maei2010toward,sutton2009fast,lee2019target,gupta2019adaptive}.
% Besides, it often provides deeper insights on RL algorithms through dynamic system perspectives.
However, its application to Q-learning has been limited due to the presence of the max operator, which makes the associated ODE model a complex nonlinear system. In contrast, the associated ODE of TD learning for policy evaluation is linear, whose asymptotic stability is much easier to analyze in general. While~\cite{borkar2000ode} gave the convergence proof of Q-learning based on a nonlinear ODE model, to the authors' knowledge, substantial analysis is required to prove the stability of the corresponding nonlinear ODE~\citep{borkar1997analog} by using the max-norm contraction of the Bellman operator. It should also be noted that the result in~\cite{borkar2000ode} only applies to synchronous Q-learning, where every state-action pair should be visited one time at each iteration, instead of the commonly used asynchronous Q-learning. Last but not least, the stability analysis does not immediately extend to other Q-learning variants, such as double Q-learning~\citep{hasselt2010double}, averaging Q-learning~\citep{lee2019target}, and Q-learning with linear function approximation.

In this paper, we provide a simple and unified framework to analyze Q-learning and its variants through switched linear system (SLS) models~\citep{liberzon2003switching} of the associated ODE. SLSs are an important class of nonlinear hybrid systems, where the system dynamics matrix varies within a finite set of subsystem matrices (or modes) according to a switching signal. The study of SLSs has attracted much attention in the past years and their stability behaviors have been well established in the literature; see~\cite{lin2009stability} and~\cite{liberzon2003switching} for comprehensive surveys. Our main contributions are summarized as follows:
\begin{enumerate}
\item  For a number of Q-learning algorithms such as the asynchronous Q-learning, we show that the nonlinear ODE models associated with these algorithms can be characterized as switched linear systems, or more precisely, switched affine systems with a  state-feedback switching policy.

\item We construct both upper and lower comparison systems of the corresponding switched affine systems, and prove their asymptotic stability based on existing switching system theory and comparison principles. As a result of the Borkar and Meyn theorem~\citep{borkar2000ode}, we obtain the asymptotic convergence of these Q-learning algorithms.

\item We also extend the approach to analyze the averaging Q-learning~\citep{lee2019target}. To our best knowledge, the convergence of averaging is analyzed for the first time in the literature.

\item Lastly, we also examine Q-learning with linear function approximation and derive a new sufficient condition to ensure its convergence based on the switching system theory. We show that, under specific assumptions, this new condition is weaker than the condition provided in~\cite{melo2008analysis}.

\end{enumerate}
It is worth mentioning that a number of recent works establish the convergence analysis of reinforcement learning algorithms based on their connections to control theory. For example,~\cite{srikant2019finite} provides the finite sample bound of TD learning based on Lyapunov stability theory for linear ODE. \cite{chen2019performance,yang2019theoretical} provided the analysis for Q-learning with linear function approximation through analyzing the associated nonlinear ODE,  assuming the stability of the system, e.g., under the condition provided in~\cite{melo2008analysis}. Note that the new sufficient condition that guarantees its convergence in this paper is weaker than that used in~\cite{chen2019performance}. Another closely related work is~\cite{hu2019characterizing}, which explores the connection between temporal difference (TD) learning and dynamic systems, particularly the Markov jump linear systems (MJLS). Note that MJLS cannot be used to characterize the nonlinear dynamics of Q-learning. Instead, we resort to switched linear systems with state-feedback switching policies. To the authors' best knowledge, this is the first work that provides the connection between Q-learning and switching system theory. Our new ODE approach based on switched linear system can be used as a viable alternative to prove the stability of the associated ODE of various reinforcement learning algorithms as well as their asymptotic (and potentially non-asymptotic) convergence.

 % It should be noticed that most existing results on the non-asymptotic convergence rate of Q-learning (and its variants) are built on the martingale-based approach, rather than the ODE approach; see e.g.,~\cite{even2003learning,azar2011speedy,zou2019finite}.

\paragraph{Notation.} The following notation is adopted: ${\mathbb R}^n $ is the
$n$-dimensional Euclidean space; ${\mathbb R}^{n \times m}$
denotes the set of all $n \times m$ real matrices; ${\mathbb R}_+^n$
is the sets of vectors with nonnegative real elements; $A^T$ denotes the transpose of matrix
$A$; $I_n$ is the $n \times n$ identity matrix; $I$ stands for
the identity matrix with appropriate dimension; $|{\cal S}|$ means the cardinality of the set for any finite set ${\cal
S}$; ${\mathbb E}[\cdot]$ is the expectation operator;
${\mathbb P}[\cdot]$ means the probability of an event; $[M]_{ij}$ indicates the
element in $i$-th row and $j$-th column for any matrix $M$; $e_j,j\in \{ 1,2,\ldots,n\}$, is
the $j$-th basis vector (all components are $0$ except for the
$j$-th component which is $1$) of appropriate dimensions. $\otimes$ stands for the Kronecker product between two matrices.

%%%%%%%%%%%%%%%%%%%%%%%%%%%%%%%%%%%%%%%%%%%%%%%%%%%%%%%%%%%%%%%%%%%%%%%%%%%%%%%%
\section{Preliminaries}

\subsection{Markov decision problem}
We consider the infinite-horizon (discounted) Markov decision problem (MDP), where the agent sequentially takes actions to maximize cumulative discounted rewards. In a Markov decision process with the state-space ${\cal
S}:=\{ 1,2,\ldots ,|{\cal S}|\}$ and action-space ${\cal A}:= \{1,2,\ldots,|{\cal A}|\}$, the decision maker selects an action $a \in {\cal A}$ with the current state $s$, then the state
transits to $s'$ with probability $P_a(s,s')$, and the transition incurs a random
reward $r_a(s,s')$, where $P_a\in {\mathbb
R}^{|{\cal S}| \times |{\cal S}|},a \in {\cal A}$, $P_a
(s,s')$ is the state transition probability from the current state
$s\in {\cal S}$ to the next state $s' \in {\cal S}$ under action
$a \in {\cal A}$, and $r_a(s,s')$ is the reward random variable conditioned on $a
\in {\cal A},s,s'\in {\cal S}$ with its expectation ${\mathbb E}[r_a(s,s')|s,a,s']= R_a(s,s')$. A deterministic policy, $\pi :{\cal S} \to {\cal A}$, maps a state $s \in {\cal S}$ to an action $\pi(s)\in {\cal A}$. The Markov decision problem (MDP) is to find a deterministic optimal policy, $\pi^*$, such that the cumulative discounted rewards over infinite time horizons is
maximized, i.e.,
\begin{align*}
%=======================================================================================
\pi^*:= \argmax_{\pi\in \Theta} {\mathbb E}\left[ \left.\sum_{k=0}^\infty {\alpha^k r_{a_k}(s_k,s_{k+1})}\right|\pi\right],
%=======================================================================================
\end{align*}
where $\gamma \in [0,1)$ is the
discount factor, $\Theta$ is the set of all admissible deterministic policies, $(s_0,a_0,s_1,a_1,\ldots)$ is a state-action trajectory generated by the Markov chain under policy $\pi$, and ${\mathbb E}[\cdot|\cdot,\pi]$ is an expectation conditioned on the policy $\pi$. The Q-function under policy $\pi$ is defined as
\begin{align*}
%=======================================================================================
&Q^{\pi}(s,a)={\mathbb E}\left[ \left. \sum_{k=0}^\infty {\gamma^k r_{a_k}(s_k,s_{k+1})} \right|s_0=s,a_0=a,\pi \right],\quad s\in {\cal S},a\in {\cal A},
%=======================================================================================
\end{align*}
and the corresponding optimal Q-function is defined as $Q^*(s,a)=Q^{\pi^*}(s,a)$ for all $s\in {\cal S},a\in {\cal A}$. Once $Q^*$ is known, then an optimal policy can be retrieved by $\pi^*(s)=\argmax_{a\in {\cal A}}Q^*(s,a)$.

\subsection{Basics of nonlinear system theory}
Consider the nonlinear system
\begin{align}
%================================================================
\frac{d}{dt}x_t=f(x_t),\quad x_0=z,\quad t\in {\mathbb R}_+,\label{eq:nonlinear-system}
%================================================================
\end{align}
where $x_t\in {\mathbb R}^n$ is the state and $f:{\mathbb R}^n \to {\mathbb R}^n$ is a nonlinear mapping. For simplicity, we assume that the solution to~\eqref{eq:nonlinear-system} exists and is unique. In fact, this holds true so long as the mapping $f$ is globally Lipschitz continuous.
\begin{lemma}[{\cite[Theorem~3.2]{khalil2002nonlinear}}]\label{lemma:existence}
%================================================================
Consider the nonlinear system~\eqref{eq:nonlinear-system} and assume that $f$ is globally Lipschitz continuous, i.e.,
\begin{align}
%================================================================
&\|f(x)-f(y)\|\le L \|x-y\|,\quad \forall x,y \in {\mathbb R}^n,
%================================================================
\end{align}
for some $L>0$ and norm $\|\cdot\|$, then it has a unique solution $x(t)$ for all $t\geq 0$ and $x(0) \in {\mathbb R}^n$.
%================================================================
\end{lemma}

An important concept in dealing with the nonlinear system is the equilibrium point. A point $x=x^e$ in the state space is said to be an equilibrium point of~\eqref{eq:nonlinear-system} if it has the property that whenever the state of the system starts at $x^e$, it will remain at $x^e$~\citep{khalil2002nonlinear}. For~\eqref{eq:nonlinear-system}, the equilibrium points are the real roots of the equation $f(x)=0$. The equilibrium point $x^e$ is said to be globally asymptotically stable if for any initial state $x_0 \in {\mathbb R}^n$, $x_t \to x^e$ as $t \to \infty$. Now, we provide a vector comparison principle~\citep{walter1998ordinary,hirsch2006monotone,platzer2018vector} for multi-dimensional O.D.E., which will play a central role in the analysis below. We first introduce the quasi-monotone increasing function, which is a necessary prerequisite for the comparison principle.
\begin{definition}[Quasi-monotone function]\label{def:quasi-monotone}
%================================================================
A vector-valued function $f:{\mathbb R}^n \to {\mathbb R}^n$ with $f:=\begin{bmatrix} f_1 & f_2 & \cdots & f_n\\ \end{bmatrix}^T$ is said to be quasi-monotone increasing if $f_i (x) \le f_i (y)$ holds for all $i \in \{1,2,\ldots,n \}$ and $x,y \in {\mathbb R}^n$ such that $x_i=y_i$ and $x_j\le y_j$ for all $j\neq i$.
%================================================================
\end{definition}

An example of a quasi-monotone increasing function $f$ is $f(x)=Ax$ where $A$ is a Metzler matrix, which implies the off-diagonal elements of $A$ are nonnegative. Now, we introduce a vector comparison principle. For completeness, we provide a simple proof tailored to our interests.
\begin{lemma}[Vector comparison principle {\cite[Theorem~3.2]{hirsch2006monotone}}]\label{lemma:comparision-principle}
%================================================================
 Suppose that $\overline{f}$ and $\underline{f}$ are globally Lipschitz continuous. Let $v_t$ be a solution of the system
\begin{align*}
%================================================================
\frac{d}{dt}x_t=\overline{f}(x_t),\quad x_0\in {\mathbb R}^n,\forall \quad t  \geq 0,
%================================================================
\end{align*}
assume that $\overline{f}$ is quasi-monotone increasing, and let $v_t$ be a solution of the system
\begin{align}
%================================================================
\frac{d}{dt} v_t = \underline{f}(v_t), \quad v_0 < x_0, \quad \forall t \geq 0,\label{eq:lower-system}
%================================================================
\end{align}
where $\underline{f}(v) \le \overline{f}(v)$ holds for any $v \in {\mathbb R}^n$. Then, $v_t\leq x_t$  for all $t \geq 0$.
%================================================================
\end{lemma}
\begin{proof}
%================================================================
Instead of~\eqref{eq:lower-system}, first consider
\begin{align*}
%================================================================
\frac{d}{dt} v_{\varepsilon}(t) = \underline{f}(v_{\varepsilon}(t))-\varepsilon {\bf 1}_n,\quad v_{\varepsilon}(0) < x(0), \quad \forall t \geq 0
%================================================================
\end{align*}
where $\varepsilon>0$ is a sufficiently small real number and ${\bf 1}_n$ is a vector where all elements are ones, where we use a different notation for the time index for convenience. Suppose that the statement is not true, and let
\begin{align*}
%================================================================
t^* :=\inf \{t\ge 0:\exists i\,\,{\rm such\,\,that}\,\,v_{\varepsilon,i}(t)>x_{i}(t)\}<\infty,
%================================================================
\end{align*}
and let $i$ be such index. By the definition of $t^*$, we have that $v_{\varepsilon,i} (t^* ) = x_{i}(t^* )$ and $v_{\varepsilon ,j} (t^* ) \le x_j (t^* )$ for any $j\neq i$. Then, since $\overline{f}$ is quasi-monotone increasing, we have
\begin{align}
%================================================================
\overline{f}_i(v_\varepsilon (t^*))\le \overline{f}_i(x(t^*)).\label{eq:5}
%================================================================
\end{align}

On the other hand, by the definition of $t^*$, there exists a small $\delta >0$ such that
\begin{align*}
%================================================================
v_{\varepsilon,i} (t^*+\Delta t)>x_i (t^*+\Delta t)
%================================================================
\end{align*}
for all $0<\Delta t<\delta$. Dividing both sides by $\Delta t$ and taking the limit $\Delta t \to 0$, we have
\begin{align}
%================================================================
\dot v_{\varepsilon,i}(t^*)\ge\dot x_i(t^*)=\overline{f}_i(x(t^*)).\label{eq:2}
%================================================================
\end{align}

By the hypothesis, it holds that
\begin{align*}
%================================================================
\frac{d}{dt}v_\varepsilon(t)=\underline{f}(v_\varepsilon(t))-\varepsilon {\bf 1}_n < \underline{f}(v_\varepsilon(t))\le \overline{f}(v_\varepsilon(t))
%================================================================
\end{align*}
holds for all $t\geq 0$. The inequality implies $\dot v_{\varepsilon,i}(t)<\overline{f}_i (v_\varepsilon (t))$, which in combination with~\eqref{eq:2} leads to $\overline{f}_i (v_\varepsilon(t^*))>\overline{f}_i(x(t^*))$. However, it contradicts with~\eqref{eq:5}. Therefore, $v_\varepsilon(t)\le x(t)$ holds for all $t\geq 0$. Since the solution $v_\varepsilon(t)$ continuously depends on $\varepsilon >0$ \citep[Chap.~13]{walter1998ordinary}, taking the limit $\varepsilon \to 0$, we conclude $v_0(t)\le x(t)$ holds for all $t\geq 0$. This completes the proof.

%================================================================
\end{proof}

\subsection{Switching system theory}

Consider the particular nonlinear system, the \emph{switched linear  system},
\begin{align}
%=============================================================================================================
&\frac{d}{dt} x_t=A_{\sigma_t} x_t,\quad x_0=z\in {\mathbb
R}^n,\quad t\in {\mathbb R}_+,\label{eq:switched-system}
%=============================================================================================================
\end{align}
where $x_t \in {\mathbb R}^n$ is the state,  $\sigma\in {\mathcal M}:=\{1,2,\ldots,M\}$ is called the mode,  $\sigma_t \in
{\mathcal M}$ is called the switching signal, and $\{A_\sigma,\sigma\in {\mathcal M}\}$ are called the subsystem matrices. The switching signal can be either arbitrary or controlled by the user under a certain switching policy. Especially, a state-feedback switching policy is denoted by $\sigma(x_t)$. To prove the global asymptotic stability of the switching system, we will use a fundamental algebraic stability condition of switching systems reported in~\cite{lin2009stability}.
More comprehensive surveys of stability of switching systems can be found in~\cite{lin2009stability} and~\cite{liberzon2003switching}.
\begin{lemma}[{\cite[Theorem~8]{lin2009stability}}]\label{lemma:fundamental-stability-lemma}
%================================================================
The origin of the linear switching system~\eqref{eq:switched-system} is the unique globally asymptotically stable equilibrium point under arbitrary switchings, $\sigma_t$, if and only if there exist a full column rank matrix , $L\in {\mathbb R}^{m\times n}$, $m\geq n$,
and a family of matrices, $\bar A_\sigma \in {\mathbb R}^{m\times n}, \sigma \in {\cal M}$, with
the so-called strictly negative row dominating diagonal condition, i.e., for each $\bar A_\sigma, \sigma \in {\cal M}$,
its elements satisfying
\begin{align*}
%================================================================
&[\bar A_\sigma]_{ii}+\sum_{j \in \{ 1,2, \ldots ,n\} \backslash \{ i\} } {|[\bar A_\sigma]_{ij}}|<0,\quad \forall i \in \{ 1,2, \ldots ,m\},
%================================================================
\end{align*}
where $[\cdot]_{ij}$ is the $(i,j)$-element of a matrix $(\cdot)$, such that the matrix relations
\begin{align*}
%================================================================
&L A_\sigma=\bar A_\sigma L,\quad\forall\sigma\in {\cal M},
%================================================================
\end{align*}
are satisfied.
%================================================================
\end{lemma}

\subsection{ODE-based stochastic approximation}\label{sec:ODE-stochastic-approximation}

Due to its generality, the convergence analyses of many RL algorithms rely on the ODE (ordinary differential equation) approach~\citep{bhatnagar2012stochastic,kushner2003stochastic}. It analyzes convergence of general stochastic recursions by examining stability of the associated ODE model based on the fact that the stochastic recursions with diminishing step-sizes approximate the corresponding ODEs in the limit. One of the most popular approach is based on the Borkar and Meyn theorem~\citep{borkar2000ode}. We now briefly introduce the Borkar and Meyn's ODE approach for analyzing convergence of the general stochastic recursions
\begin{align}
%================================================================
&\theta_{k+1}=\theta_k+\alpha_k (f(\theta_k)+\varepsilon_{k+1})\label{eq:general-stochastic-recursion}
%================================================================
\end{align}
where $f:{\mathbb R}^n \to {\mathbb R}^n$ is a nonlinear mapping. Basic technical assumptions are given below.
\begin{assumption}\label{assumption:1}
%================================================================
$\,$\begin{enumerate}
%================================================================
\item The mapping $f:{\mathbb R}^n  \to {\mathbb R}^n$ is
globally Lipschitz continuous and there exists a function
$f_\infty:{\mathbb R}^n\to {\mathbb R}^n$ such that
\begin{align*}
%================================================================
&\lim_{c\to \infty}\frac{f(c x)}{c}= f_\infty(x),\quad \forall x \in {\mathbb R}^n.
%================================================================
\end{align*}

\item The origin in ${\mathbb R}^n$ is an asymptotically stable
equilibrium for the ODE $\dot x_t=f_\infty (x_t)$.

\item There exists a unique globally asymptotically stable equilibrium
$\theta^e\in {\mathbb R}^n$ for the ODE $\dot x_t=f(x_t)$, i.e., $x_t\to\theta^e$ as $t\to\infty$.

\item The sequence $\{\varepsilon_k,{\cal G}_k,k\ge 1\}$ with ${\cal G}_k=\sigma(\theta_i,\varepsilon_i,i\le k)$
is a Martingale difference sequence. In addition, there exists a constant $C_0<\infty $ such that for any initial $\theta_0\in
{\mathbb R}^n$, we have ${\mathbb E}[\|\varepsilon_{k+1} \|^2 |{\cal G}_k]\le C_0(1+\|\theta_k\|^2),\forall k \ge 0$.

\item The step-sizes satisfy
$\alpha_k>0, \sum_{k=0}^\infty {\alpha_k}=\infty, \sum_{k=0}^\infty{\alpha_k^2}<\infty$.
% ~\eqref{eq:step-size-rule}.
%================================================================
\end{enumerate}
%================================================================
\end{assumption}

\begin{lemma}[{\cite[Borkar and Meyn theorem]{borkar2000ode}}]\label{lemma:Borkar}
%================================================================
Suppose that~\cref{assumption:1} holds. For any initial $\theta_0\in
{\mathbb R}^n$, $\sup_{k\ge 0} \|\theta_k\|<\infty$
with probability one. In addition, $\theta_k\to\theta^e$ as
$k\to\infty$ with probability one.
%================================================================
\end{lemma}
The Borkar and Meyn theorem states that under~\cref{assumption:1}, the stochastic process $(\theta_k)_{k=0}^\infty$ generated by~\eqref{eq:general-stochastic-recursion} is bounded and converges to $\theta^e$ almost surely.

\section{Revisit Q-learning}
We now briefly review the standard Q-learning and its convergence. Recall that the standard Q-learning updates
\begin{equation}
%=======================================================================================
\begin{split}
Q_{k+1}(s_k,a_k)=&Q_k(s_k,a_k)+\\
&\alpha_k(s_k,a_k)\left\{r_{a_k}(s_k,s_{k+1})+\gamma \max_{a \in {\cal A}} Q_k (s_{k+1},a_k)-Q_k(s_k,a_k)\right\},
\end{split}
%=======================================================================================
\end{equation}
where  $0 \le \alpha_k(s,a) \le 1$ is called the learning rate associated
with the state-action pair $(s,a)$ at iteration $k$. This value is assumed to be zero if
$(s,a)\ne (s_k,a_k)$. If
\begin{align*}
%=======================================================================================
&\sum_{k=0}^\infty{\alpha_k(s,a)}= \infty ,\quad \sum_{k=0}^\infty {\alpha_k^2 (s,a)}<\infty,
%=======================================================================================
\end{align*}
and every state-action pair is visited infinitely often, then the iterate is guaranteed to converge to $Q^*$ with probability one. Note that the state-action can be visited arbitrarily, which is more general than stochastic visiting rules.

To analyze the convergence based on the switching system model, we consider the stronger assumption that $\{(s_k,a_k)\}_{k=0}^{\infty}$ is a sequence of i.i.d. random variables with an identical underlying probability distribution, $d_a(s),s\in {\cal S},a \in {\cal A}$, of the state and action pair $(s,a)$. For example, this can be a stationary state-action distribution under the behavior policy, where the behavior policy is the policy by which the RL agent actually behaves to collect experiences. This assumption is common in the ODE approaches for Q-learning and TD-learning~\citep{sutton1988learning}. This assumption can be relaxed by considering a time-varying distribution. However, this direction will not be addressed in this paper to simplify the presentation of the proofs.

Throughout the paper, we assume that
\begin{assumption}\label{assumption:positive-distribution}
%=======================================================================================
$d_a(s)> 0$ holds for all $s\in {\cal S},a \in {\cal A}$.
%=======================================================================================
\end{assumption}
 Under this assumption, the modified standard Q-learning is given in~\cref{algo:standard-Q-learning}. Compared to the original version, the step-size $\alpha_k$ does not depend on the state-action pair in this version. With a suitable choice on the step-size,~\cref{algo:standard-Q-learning} converges to the optimal $Q^*$ with probability one.
\begin{algorithm}[t]
%=========================================================================================
\caption{Standard Q-Learning}
  \begin{algorithmic}[1]
    \State Initialize $Q_0 \in {\mathbb R}^{|{\cal S}||{\cal A}|}$ randomly.
    \For{iteration $k=0,1,\ldots$}
    	\State Sample $(s,a)\sim d_a(s)$
        \State Sample $s'\sim P_a(s,\cdot)$ and $r_a(s,s')$
        \State Update $Q_{k+1}(s,a)=Q_k(s,a)+\alpha_k \{r_a(s,s')+\gamma\max_{a \in {\cal A}} Q_k(s',a)-Q_k (s,a)\}$
    \EndFor

  \end{algorithmic}\label{algo:standard-Q-learning}
%=========================================================================================
\end{algorithm}

\begin{theorem}\label{thm:Q-learning-convergence}
%================================================================
Assume that the step-sizes satisfy
\begin{align}
%================================================================
&\alpha_k>0,\quad \sum_{k=0}^\infty {\alpha_k}=\infty,\quad \sum_{k=0}^\infty{\alpha_k^2}<\infty.\label{eq:step-size-rule}
%================================================================
\end{align}
Then, $Q_k \to Q^*$ with probability one.
%================================================================
\end{theorem}

% In the next section, we prove~\cref{thm:Q-learning-convergence} based on the ODE approach, where the ODE is described by a linear switching system.

%%%%%%%%%%%%%%%%%%%%%%%%%%%%%%%%%%%%%%%%%%%%%%%%%%%%%%%%%%%%%%%%%%%%%%%%%%%%%%%%
\section{Convergence of Q-learning from Switching System Theory}

In this section, we study a switching system-based ODE model of Q-learning and prove the convergence of Q-learning in~\cref{thm:Q-learning-convergence} based on the switching system analysis. We first introduce the following compact notations:
\begin{align*}
%================================================================
P:=& \begin{bmatrix}
   P_1\\
   \vdots\\
   P_{|{\cal A}|}\\
\end{bmatrix} \in {\mathbb R}^{|{\cal S}| \times |{\cal S}||{\cal A}|} ,\quad R:= \begin{bmatrix}
   R_1 \\
   \vdots \\
   R_{|{\cal A}|} \\
\end{bmatrix} \in {\mathbb R}^{|{\cal S}||{\cal A}|},\quad Q:= \begin{bmatrix}
   Q_1\\
  \vdots\\
   Q_{|{\cal A}|}\\
\end{bmatrix}\in {\mathbb R}^{|{\cal S}||{\cal A}|},\\
%================================================================
D_a:=& \begin{bmatrix}
   d_a(1) & & \\
   & \ddots & \\
   & & d_a(|{\cal S}|)\\
\end{bmatrix}\in {\mathbb R}^{|{\cal S}| \times |{\cal S}|},\quad D:=\begin{bmatrix}
   D_1 & & \\
    & \ddots  & \\
    & & D_{|{\cal A}|} \\
\end{bmatrix}\in {\mathbb R}^{|{\cal S}||{\cal A}| \times |{\cal S}||{\cal A}|},
%================================================================
\end{align*}
where $Q_a= Q(\cdot,a)\in {\mathbb R}^{|{\cal S}|},a\in {\cal A}$ and $R_a(s):={\mathbb E}[r_a(s,s')|s,a]$. Note that under~\cref{assumption:positive-distribution}, $D$ is a nonsingular diagonal matrix with strictly positive diagonal elements.

\subsection{Q-learning as switched linear system}
Using the notation introduced, the update in~\cref{algo:standard-Q-learning} can be rewritten as
\begin{equation*}
\begin{split}
%================================================================
Q_{k+1}=Q_k&+\alpha_k \{(e_a\otimes e_s )(e_a\otimes e_s)^T R\\
&+\gamma(e_a\otimes e_s)(e_{s'})^T \max_{a\in {\cal A}}Q_k(\cdot,a)-(e_a  \otimes e_s)(e_a\otimes e_s)^T Q_k\},
%================================================================
\end{split}
\end{equation*}
where $e_s \in {\mathbb R}^{|{\cal S}|}$ and $e_a \in {\mathbb R}^{|{\cal A}|}$ are $s$-th basis vector (all components are $0$ except for the $s$-th component which is $1$) and $a$-th basis vector, respectively. For any deterministic policy, $\pi:{\cal S}\to {\cal A}$, we define the corresponding distribution vector
\begin{align*}
%================================================================
&\vec{\pi}(s):=e_{\pi(s)}\in \Delta_{|{\cal S}|},
%================================================================
\end{align*}
where $\Delta_{|{\cal S}|}$ is the set of all probability distributions over ${\cal S}$, and the matrix
\begin{align*}
%================================================================
\Pi_\pi:=\begin{bmatrix}
   \vec{\pi}(1)^T \otimes e_1^T\\
   \vec{\pi}(2)^T \otimes e_2^T\\
    \vdots\\
   \vec{\pi}(|S|)^T \otimes e_{|S|}^T \\
\end{bmatrix}\in {\mathbb R}^{|{\cal S}| \times |{\cal S}||{\cal A}|}.
%================================================================
\end{align*}
% Then, the update can be simplified as
% \begin{align*}
% %================================================================
% Q_{k+1}=Q_k+\alpha_k \{(e_a\otimes e_s)(e_a\otimes e_s)^T R+\gamma(e_a\otimes e_s )(e_{s'})^T \Pi_{\pi_{Q_k}} Q_k -(e_a\otimes e_s)(e_a\otimes e_s)^T Q_k\},
% %================================================================
% \end{align*}
Denoting $\pi_Q(s):=\argmax_{a\in {\cal A}} e_s^T Q_a\in {\cal A}$, the above update can be further expressed as
\begin{align}
%================================================================
Q_{k+1}=Q_k+\alpha_k \{DR+\gamma DP\Pi_{\pi_{Q_k}}Q_k-DQ_k +\varepsilon_{k+1}\},\label{eq:1}
%================================================================
\end{align}
where
\begin{align*}
%================================================================
\varepsilon_{k+1}=&(e_a\otimes e_s )(e_a\otimes e_s)^T R+\gamma (e_a\otimes e_s )(e_{s'})^T \Pi_{\pi_{Q_k}}Q_k\\
%================================================================
&-(e_a\otimes e_s)(e_a\otimes e_s)^T Q_k-(DR+\gamma DP\Pi_{\pi_{Q_k}}Q_k-DQ_k).
%================================================================
\end{align*}
We note that, for any $\pi \in \Theta$, $P\Pi_{\pi}$ is the state-action pair transition probability matrix under the deterministic policy $\pi$. Using the Bellman equation $(\gamma DP\Pi_{\pi_{Q^*}}-D)Q^*+DR=0$,~\eqref{eq:1} is rewritten by
\begin{equation}\label{eq:Q-learning-stochastic-recursion-form}
\begin{split}
%================================================================
(Q_{k+1}-Q^*)=(Q_k-Q^*)&+\alpha_k \{(\gamma DP\Pi_{\pi_{Q_k}}-D)(Q_k-Q^*)\\
&+\gamma DP(\Pi_{\pi_{Q_k}}-\Pi_{\pi_{Q^*}})Q^*+\varepsilon_{k+1}\}.
%================================================================
\end{split}
\end{equation}

As discussed in~\cref{sec:ODE-stochastic-approximation}, the convergence of~\eqref{eq:Q-learning-stochastic-recursion-form} can be analyzed by evaluating the stability of the corresponding continuous-time ODE
\begin{equation}\label{eq:4}
\begin{array}{rl}
%================================================================
\frac{d}{dt}(Q_t-Q^*)&=(\gamma DP\Pi_{\pi_{Q_t}}-D)(Q_t-Q^*)+\gamma DP(\Pi_{\pi_{Q_t}}-\Pi_{\pi_{Q^*}})Q^*,\\
 Q_0-Q^*&=z \in {\mathbb R}^{|{\cal S}||{\cal A}|},
%================================================================
\end{array}
\end{equation}
which is a linear switching system. More precisely, if we define a one-to-one map $\psi :\Theta \to \{1,2,\ldots ,|\Theta |\}$, where $\Theta$ is the set of all deterministic policies, $x_t:= Q_t-Q^*$, and
\begin{align*}
%================================================================
(A_{\psi(\pi)},b_{\psi(\pi)}):=(\gamma DP\Pi_\pi -D,\gamma DP(\Pi_{\pi}-\Pi_{\pi_{Q^*}})Q^*)
%================================================================
\end{align*}
for all $\pi\in \Theta$, then~\eqref{eq:4} can be represented by the affine switching system
\begin{align}
%================================================================
\frac{d}{dt}x_t= A_{\sigma(x_t)}x_t + b_{\sigma(x_t)},\quad x_0=z\in {\mathbb R}^{|{\cal S}||{\cal A}|},\label{eq:swithcing-system-form}
%================================================================
\end{align}
where, $\sigma: {\mathbb R}^{|{\cal S}||{\cal A}|}\to \{1,2,\ldots ,|\Theta |\}$ is a state-feedback switching policy defined by $\sigma(x_t):=\psi(\pi_{Q_t})$, $\pi_{Q_t}(s)=\argmax_{a\in {\cal A}} e_s^T Q_{t,a}$.

Since~\eqref{eq:swithcing-system-form} is a switching system with a state-feedback switching policy, it may cause arbitrary switching behaviors. It is unclear whether its solution exists over all $t\geq 0$ and whether the solution is unique. We establish the existence and uniqueness of its solution, which follows from the global Lipschitz continuity of the affine mapping.
\begin{proposition}\label{prop:Lipschitz-f}
%================================================================
Define
\begin{align*}
%================================================================
f(x)=(\gamma DP\Pi_{\pi_x}-D)x+DR.
%================================================================
\end{align*}
Then, $f$ is globally Lipschitz continuous.
%================================================================
\end{proposition}
\begin{proof}
%================================================================
The proof is completed by the inequalities
\begin{align*}
%================================================================
\|f(x)-f(y)\|_\infty =&\| (\gamma DP\Pi_{\pi_x}-D)x-(\gamma DP\Pi_{\pi_y}-D)y \|_\infty\\
%================================================================
\le& \| \gamma DP \|_\infty \|\Pi_{\pi_x}x-\Pi_{\pi_y}y\|_\infty + \|D\|_\infty \|x-y\|_\infty\\
%================================================================
=& \|\gamma DP\|_\infty \max_{s\in {\cal S}}|\max_{a\in {\cal A}} x_a(s)-\max_{a \in {\cal A}}y_a(s)| + \|D\|_\infty \|x-y\|_\infty\\
%================================================================
\le& \|\gamma DP\|_\infty  \max_{s\in {\cal S}} \max_{a\in {\cal A}} |x_a(s)-y_a(s)| + \|D\|_\infty \|x-y\|_\infty\\
%================================================================
=& \|\gamma DP\|_\infty \|x-y\|_\infty + \|D\|_\infty \|x-y\|_\infty\\
%================================================================
\le& \left(\|\gamma DP\|_\infty + \|D\|_\infty \right)\|x - y\|_\infty,
%================================================================
\end{align*}
indicating that $f$ is globally Lipschitz continuous with respect to the $\|\cdot\|_\infty$ norm. This completes the proof.
%================================================================
\end{proof}

Invoking~\cref{lemma:existence}, we then have the following result
\begin{proposition}\label{prop:existence-of-solution}
%================================================================
The solution of the switching system~\eqref{eq:swithcing-system-form} exists and is unique for all $t\geq 0$ and $x(0) \in {\mathbb R}^n$.
%================================================================
\end{proposition}

\begin{remark}
The ODE model of the Q-learning in~\cite{borkar1997analog} is
\begin{align}
%================================================================
\frac{d}{dt}Q_t=(\gamma P\Pi_{Q_t}-I)Q_t+R={\bf T}Q_t-Q_t,\label{eq:ode-operator}
%================================================================
\end{align}
where ${\bf T}Q_t=\gamma P\Pi_{Q_t} Q_t-R$ is the Bellman operator. The approach in~\cite{borkar1997analog} proves that an ODE of the form~\eqref{eq:ode-operator} is globally asymptotically stable if the operator ${\bf T}$ is a contraction with respect to some norm, and the unique fixed point of ${\bf T}$ is the unique equilibrium point of~\eqref{eq:ode-operator}. The ODE in~\eqref{eq:4} is reduced to~\eqref{eq:ode-operator} with some modifications and setting $D =I$. The Q-learning corresponding to~\eqref{eq:ode-operator} should be a synchronous version~\citep{even2003learning}, i.e., all the state-action pair is visited one time at every iteration of the recursion. To explain this, consider a stochastic approximation $\hat P$ of $P$ such that ${\mathbb E}[\hat P] = P$. One possible approach to construct such $\hat P$ is
\begin{align*}
%================================================================
&{\mathbb E}\left[ \sum_{i\in {\cal S},j \in {\cal A}}{(e_j\otimes e_i)e_{s'_{ij}}^T} \right]=\sum_{s\in {\cal S},a \in {\cal A}}{{\mathbb E}[(e_a  \otimes e_s)e_{s'}^T]}=P,
%================================================================
\end{align*}
where $s'_{ij}$ is an independently sampled next state for the current state $i\in {\cal S}$. This implies that to construct $\hat P$, each state should be visited at least one time. Similarly, to obtain a stochastic approximation whose mean is ${\bf T}Q_t$ in~\eqref{eq:ode-operator}, every state should be visited at least one time at every iteration.

Under the assumption of samplings from an identical distribution, the corresponding ODE is
\begin{align*}
%================================================================
\frac{d}{dt}Q_t=(\gamma DP\Pi_{Q_t}-D)Q_t + DR = {\bf F}Q_t -Q_t,
%================================================================
\end{align*}
where ${\bf F}Q_t:=\gamma DP\Pi_{Q_t}Q_t-DQ_t+DR+Q_t$. However, it is not trivial to prove if ${\bf F}$ is a contraction in this case. In this sense, the approach in~\cite{borkar1997analog} cannot be directly applied for asymptotic stability of the general system~\eqref{eq:4}.
%================================================================
\end{remark}

\subsection{Stability analysis}

Note that the proving the global asymptotic stability of~\eqref{eq:swithcing-system-form} without the affine term is relevantly straightforward based on existing results, e.g.,~\citet[Theorem~8]{lin2009stability}. However, with the affine term, the proof is no longer trivial with the existing approaches in switching system theories. In what follows, we show that by exploiting the special structure of the switching system and policy associated with the Q-learning update rule, the global asymptotic stability can still be proved. We first establish the asymptotic stability of the corresponding linear switching system.
\begin{lemma}\label{lem:linear-system-stability} Consider the affine switching system~\eqref{eq:swithcing-system-form}. The origin of the associated linear switching system
\begin{align*}
%================================================================
&\frac{d}{dt}x_t=A_{\sigma_t}x_t,
%================================================================
\end{align*}
is the unique globally asymptotically stable equilibrium point under arbitrary switchings, $\sigma_t$.
\end{lemma}
\begin{proof}
The proof follows by applying~\cref{lemma:fundamental-stability-lemma} with $L = I,\bar A_\sigma=A_\sigma$.  In this case, the condition, $L A_\sigma=\bar A_\sigma L$ holds. It remains to prove the strictly negative row dominating diagonal property. For notational convenience, we definte $\Pi_\sigma$, $\sigma\in {\cal M}$ as $\Pi_{\pi_{Q_t^B}}$ such that $\sigma = \psi (\pi_{Q_t^B})$. Then, the property is proved by
\begin{align*}
%================================================================
[A_\sigma]_{ii}+\sum_{j \in \{1,2,\ldots ,n\} \backslash \{i\}} {|[A_\sigma]_{ij}|}=& [D]_{ii} [\gamma P\Pi_\sigma -I]_{ii}+\sum_{j\in \{ 1,2,\ldots,n\}\backslash \{ i\} } {[D]_{ii}|[\gamma P\Pi_\sigma-I]_{ij}|}\\
%================================================================
\le& [\gamma P\Pi_\sigma-I]_{ii}+\sum_{j \in \{1,2,\ldots,n\} \backslash \{ i\}} {|[\gamma P\Pi_\sigma-I]_{ij}|}\\
%================================================================
=& [\gamma P\Pi_\sigma]_{ii}-1 + \sum_{j\in \{1,2,\ldots,n\} \backslash \{i\}}{|[\gamma P\Pi_\sigma]_{ij}|}\\
%================================================================
=& [\gamma P\Pi_\sigma]_{ii}+\sum_{j\in\{1,2,\ldots,n\}\backslash \{ i\}}{|[\gamma P\Pi_\sigma]_{ij}|}-1\\
%================================================================
=&\gamma-1\\
%================================================================
<& 0,\quad\forall\sigma\in {\cal M},
%================================================================
\end{align*}
which proves the global asymptotic stability.
\end{proof}

We are now in position to prove the asymptotic stability of~\eqref{eq:swithcing-system-form} associated with Q-learning.
\begin{theorem}\label{thm:swithcing-system-stability}
%================================================================
The origin is the unique globally asymptotically stable equilibrium point of the affine switching system~\eqref{eq:swithcing-system-form}.
%================================================================
\end{theorem}
\begin{proof}
%================================================================
The basic idea of the proof is to find systems whose trajectories lower and upper bounds the trajectory of~\eqref{eq:swithcing-system-form} by the vector comparison principle. Then, by proving the asymptotic stability of the two comparison systems, we can prove the asymptotic stability of~\eqref{eq:swithcing-system-form}.
% The idea is illustrated in~\cref{fig:diagram1}.
% \begin{figure}[h!]
% %=============================================================================================================
% % \centering\includegraphics[width=11cm,height=6cm]{figure1}
% % \caption{Convergence through comparison principle}\label{fig:diagram1}
% %=============================================================================================================
% \end{figure}

Since each element of $\Pi_{\pi_{Q^*}}Q^*$ takes the maximum value across $a$, it is clear that $(\Pi_{\pi_{Q_t}}-\Pi_{\pi_{Q^*}})Q^*\le 0$ holds, where the inequality is element-wise. Moreover, since $\gamma DP$ has nonnegative elements, $\gamma DP(\Pi_{\pi_{Q_t}}-\Pi_{\pi_{Q^*}})Q^*\le 0$ holds. Therefore, we have $(\gamma D_{\beta} P\Pi_{\pi_{Q_t}}-D)(Q_t-Q^*)+\gamma D P(\Pi_{\pi_{Q_t}}-\Pi_{\pi_{Q^*}})Q^*\le (\gamma D P\Pi_{\pi_{Q_t}}-D)(Q_t-Q^*)\le (\gamma D P\Pi_{\pi_{Q_t-Q^*}}-D)(Q_t-Q^*)$ for all $t\in {\mathbb R}_+$. To proceed, define the vector functions
\begin{align*}
%================================================================
\overline{f}(y)=&(\gamma D P\Pi_{\pi_{y}}-D)y,\\
%================================================================
\underline{f}(z)=& (\gamma D P\Pi_{\pi_{z+Q^*}}-D)z+\gamma D P(\Pi_{\pi_{z+Q*}}-\Pi_{\pi_{Q^*}})Q^*,
%================================================================
\end{align*}
and consider the systems
\begin{align*}
%================================================================
\frac{d}{dt}y_t = \overline{f}(y_t),\quad y_0 > Q_0-Q^*,\\
%================================================================
\frac{d}{dt}z_t = \underline{f}(z_t),\quad z_0 = Q_0-Q^*,
%================================================================
\end{align*}
for all $t \geq 0$. To apply~\cref{lemma:comparision-principle}, we will prove that $\overline{f}$ is quasi-monotone increasing. For any $z \in {\mathbb R}^{|{\cal S}||{\cal A}|}$, consider a nonnegative vector $p \in {\mathbb R}^{|{\cal S}||{\cal A}|}$ such that its $i$th element is zero. Then, for any $i \in {\cal S}$, we have
\begin{align*}
%================================================================
e_i^T \overline{f}(z+p) =& e_i^T (\gamma DP\Pi _{z + p}  - D)(z + p)\\
%================================================================
=&\gamma e_i^T DP\Pi_{z+p}(z+p)-e_i^T Dz - e_i^T Dp\\
%================================================================
=&\gamma e_i^T DP\Pi_{z+p}(z+p)-e_i^T Dz\\
%================================================================
=&\gamma e_i^T DP \begin{bmatrix}
   \max_a (z_a(1)+p_a(1))\\
   \max_a (z_a(2)+p_a(2))\\
    \vdots \\
   \max_a (z_a(|{\cal S}|)+p_a(|{\cal S}|))\\
\end{bmatrix} -e_i^T Dz\\
%================================================================
\ge&\gamma e_i^T DP\begin{bmatrix}
   \max_a z_a(1)\\
   \max_a z_a(2)\\
    \vdots \\
   \max_a z_a(|S|)\\
\end{bmatrix} - e_i^T Dz\\
%================================================================
=& e_i^T \overline{f}(z),
%================================================================
\end{align*}
which proves the quasi-monotone increasing property, where the second line is due to $e_i^T Dp = 0$. Moreover, following similar lines of the proof of~\cref{prop:Lipschitz-f}, one can prove that $\overline{f}$ is Lipschitz continuous. Using $\underline{f}(z) =(\gamma DP\Pi_{\pi_{(z+Q^*)}}-D)(z+Q^*)+DR$ and following similar lines of the proof of~\cref{prop:Lipschitz-f}, we conclude that $\underline{f}$ is Lipschitz continuous as well. Now, by~\cref{lemma:comparision-principle}, $Q_t-Q^*\le Q_t^u-Q^*$ holds for every $t\in {\mathbb R}_+$, where $Q_t^u-Q^*$ is the solution of the switching system, which we refer to as an upper comparison system
\begin{align*}
%================================================================
\frac{d}{dt}(Q_t^u-Q^*)=(\gamma D P\Pi_{\pi_{Q_t^u}}-D)(Q_t^u-Q^*),\quad Q_0^u-Q^*> Q_0-Q^*\in {\mathbb R}^{|{\cal S}||{\cal A}|},
%================================================================
\end{align*}
By ~\cref{lem:linear-system-stability}, the origin of the above switching system is globally asymptotically stable even under arbitrary switchings. Therefore, $Q_t-Q^*$ is asymptotically upper bounded by the zero vector as $t\to\infty$.

On the other hand, we have
\begin{align*}
%================================================================
(\gamma DP\Pi_{\pi_{Q_t}}-D)(Q_t-Q^*)&+\gamma D P(\Pi_{\pi_{Q_t}}-\Pi_{\pi_{Q^*}})Q^* = (\gamma D P\Pi_{\pi_{Q_t}}-D )Q_t+ D R\\
%================================================================
\ge& (\gamma D P\Pi_{\pi_{Q^*}}-D)Q_t + D R
%================================================================
=(\gamma D P\Pi_{\pi_{Q^*}}-D)(Q_t-Q^*),
%================================================================
\end{align*}
where the first inequality is due to $\gamma D P\Pi_{\pi_{Q_t}} Q_t \ge\gamma D P\Pi_{\pi_{Q^*}} Q_t$, and the second equality uses $D Q^* =\gamma D P\Pi_{\pi_{Q^*}}Q^*+DR$. Again, define the vector functions for lower comparison parts
\begin{align*}
%================================================================
\overline{f}(y)=&(\gamma D P\Pi_{\pi_y}-D )y+ DR,\\
%================================================================
\underline{f}(z)=& (\gamma D P\Pi_{\pi_{Q^*}}-D)z + DR
%================================================================
\end{align*}
and consider the systems
\begin{align*}
%================================================================
\frac{d}{dt}y_t = \overline{f}(y_t),\quad y_0 = Q_0,\\
%================================================================
\frac{d}{dt}z_t = \underline{f}(z_t),\quad z_0 < Q_0,
%================================================================
\end{align*}
for all $t \geq 0$. To apply~\cref{lemma:comparision-principle}, we can prove that $\overline{f}$ is quasi-monotone increasing following the same lines as above. $\overline{f}$ is Lipschitz continuous by~\cref{prop:Lipschitz-f} and $\underline{f}$ is Lipschitz continuous as it is linear. Therefore, we can invoke~\cref{lemma:comparision-principle}, to prove the inequality $Q_t^l-Q^*\le Q_t-Q^*$ for all $t\geq 0$, where $Q_t^l-Q^*$ is the solution of the following linear system called the lower comparison system:
\begin{align*}
%================================================================
\frac{d}{dt}(Q_t^l-Q^*)=(\gamma D P\Pi_{Q^*}-D)(Q_t^l-Q^*),\quad Q_0^l-Q^*<Q_0-Q^* \in {\mathbb R}^{|{\cal S}||{\cal A}|},
%================================================================
\end{align*}

The origin of the above linear system is globally asymptotically stable equilibrium point by~\cref{lem:linear-system-stability}. Therefore, $Q_t-Q^*$ is asymptotically lower bounded by the zero vector as $t\to\infty$. Combining the bounds, we conclude that $Q_t-Q^*\to 0$ as $t\to\infty$. This completes the proof of~\cref{thm:swithcing-system-stability}.
%================================================================
\end{proof}

\subsection{Proof of Q-learning Convergence}\label{subsec:proof-of-Q-learning}
Based on the results, we can now apply the Borkar and Meyn theorem,~\cref{lemma:Borkar}, to prove~\cref{thm:Q-learning-convergence}. The convergence proof of Q-learning in~\cite{borkar2000ode} relies on a nonlinear ODE model, whose asymptotic stability is proved in~\cite{borkar1997analog} by using the max-norm contraction of the Bellman operator. The switching system framework in the previous section provides a simpler analysis and can be easily extended to deal with many Q-learning variants, as are given in the subsequent sections.
% Based on~\cref{lemma:Borkar}, we provide the complete proof of~\cref{thm:Q-learning-convergence} in the sequel.
\begin{proof}[Proof of~\cref{thm:Q-learning-convergence}]
%================================================================
First of all, note that the affine switching system model in~\eqref{eq:swithcing-system-form} corresponds to the ODE model, $\frac{d}{dt}x_t = f(x_t)$, that appears in~\cref{assumption:1}. The proof is completed by examining all the statements in~\cref{assumption:1}:
\begin{enumerate}
%================================================================
\item Q-learning in~\eqref{eq:Q-learning-stochastic-recursion-form} can be expressed as the stochastic recursion in~\eqref{eq:general-stochastic-recursion} with
\begin{align*}
%================================================================
f(\theta)=(\gamma D P\Pi_{\pi_\theta}-D)\theta +\gamma D P(\Pi_{\pi_\theta}-\Pi_{\pi_{Q^*}})Q^*.
%================================================================
\end{align*}

To prove the first statement of~\cref{assumption:1}, we note that
\begin{align*}
%================================================================
\frac{f(c\theta)}{c}
% =&\frac{(\gamma D P\Pi_{\pi_{c\theta}}-D) c\theta +\gamma D P(\Pi_{\pi_{c\theta}}-\Pi_{\pi_{Q^*}})Q^*}{c}\\
%================================================================
=&(\gamma DP\Pi_{\pi_\theta}-D)\theta +\frac{\gamma DP(\Pi_{\pi_\theta}-\Pi_{\pi_{Q^*}})Q^*}{c},
%================================================================
\end{align*}
where the last equality is due to the homogeneity of the policy, $\pi_{c\theta}(s)=\argmax_{a\in {\cal A}} e_s^T c\theta_a=\argmax_{a\in {\cal A}}e_s^T \theta_a$. By taking the limit, we have
\begin{align*}
%================================================================
\lim_{c \to\infty}\frac{f(c\theta)}{c}&=(\gamma DP\Pi_{\pi_\theta}-D )\theta+ \lim_{c\to\infty} \frac{\gamma D P(\Pi_{\pi_\theta}-\Pi_{\pi_{Q^*}})Q^*}{c}\\
&=(\gamma D P\Pi_{\pi_\theta}-D)\theta =f_\infty(\theta).
%================================================================
\end{align*}
Moreover, $f$ is globally Lipschitz continuous according to~\cref{prop:Lipschitz-f}. Therefore, the proof is completed.

\item The second statement of~\cref{assumption:1} follows from~\cref{lem:linear-system-stability}..

\item The third statement of~\cref{assumption:1} follows from~\cref{thm:swithcing-system-stability}.

\item Next, we prove the remaining parts. Recall that the Q-learning update is expressed as
\begin{align*}
%================================================================
&Q_{k+1}=Q_k+\alpha_k (f(Q_k)+\varepsilon_{k+1})
%================================================================
\end{align*}
with the stochastic error
\begin{align*}
%================================================================
\varepsilon_{k+1}=&(e_a\otimes e_s )(e_a\otimes e_s)^T R+\gamma (e_a\otimes e_s )(e_{s'})^T \Pi_{\pi_{Q_k}}Q_k\\
%================================================================
&-(e_a\otimes e_s)(e_a\otimes e_s)^T Q_k-(DR+\gamma DP\Pi_{\pi_{Q_k}}Q_k-DQ_k)
%================================================================
\end{align*}
and
\begin{align*}
%================================================================
f(Q)=(\gamma D P\Pi_{\pi_Q}-D)Q+\gamma D P(\Pi_{\pi_Q}-\Pi_{\pi_{Q^*}})Q^*.
%================================================================
\end{align*}

Define the history  ${\cal G}_k:=(\varepsilon_k,\varepsilon_{k-1},\ldots,\varepsilon_1,Q_k,Q_{k-1},\ldots,Q_0)$, and
the process $(M_k)_{k=0}^\infty$ with $M_k:=\sum_{i=1}^k {\varepsilon_i}$. Then, we can prove that $(M_k)_{k=0}^\infty$ is Martingale. To do so, we first prove ${\mathbb E}[\varepsilon_{k+1}|{\cal G}_k]=0$ by
\begin{align*}
%================================================================
{\mathbb E}[\varepsilon_{k+1}|{\cal G}_k]=&{\mathbb E}[(e_a\otimes e_s)(e_a\otimes e_s)^T R|{\cal G}_k]+{\mathbb E}[\gamma (e_a\otimes e_s)(e_{s'})^T \Pi_{\pi_{Q_k}} Q_k|{\cal G}_k]\\
%================================================================
&-{\mathbb E}[(e_a\otimes e_s)(e_a\otimes e_s)^T Q_k|{\cal G}_k]-{\mathbb E}[DR + \gamma DP\Pi_{\pi_{Q_k}}Q_k-DQ_k|{\cal G}_k]\\
%================================================================
=&{\mathbb E}[DR+\gamma DP\Pi_{\pi_{Q_k}}Q_k-DQ_k|{\cal G}_k]-{\mathbb E}[DR+\gamma DP\Pi_{\pi_{Q_k}}Q_k -DQ_k|{\cal G}_k]\\
%================================================================
=&0,
%================================================================
\end{align*}
where the second equality is due to the i.i.d. assumption of samples. Using this identity, we have
\begin{align*}
%================================================================
{\mathbb E}[M_{k+1}|{\cal G}_k]=& {\mathbb E}\left[ \left. \sum_{i=1}^{k+1}{\varepsilon_i} \right|{\cal G}_k\right]={\mathbb E}[\varepsilon_{k+1}|{\cal G}_k]+{\mathbb E}\left[ \left. \sum_{i=1}^k {\varepsilon_i} \right|{\cal G}_k \right]\\
%================================================================
=&{\mathbb E}\left[\left.\sum_{i=1}^k{\varepsilon_i} \right|{\cal G}_k \right]=\sum_{i=1}^k {\varepsilon_i}=M_k.
%================================================================
\end{align*}
Therefore, $(M_k)_{k=0}^\infty$ is a Martingale sequence, and $\varepsilon_{k+1} = M_{k+1}-M_k$ is a Martingale difference. Moreover, it can be easily proved that the fourth condition of~\cref{assumption:1} is satisfied by algebraic calculations. Therefore, the fourth condition is met.
%================================================================
\end{enumerate}
%================================================================
\end{proof}

\section{Numerical examples demonstrating~\cref{thm:swithcing-system-stability}}

Consider an MDP with ${\cal S}=\{1,2\}$, ${\cal A}=\{1,2\}$, $\gamma = 0.9$,
\begin{align*}
%================================================================
&P_1=\begin{bmatrix}
   0.2 & 0.8\\
   0.3 & 0.7 \\
\end{bmatrix},\quad P_2=\begin{bmatrix}
   0.5 & 0.5 \\
   0.7 & 0.3\\
\end{bmatrix},
%================================================================
\end{align*}
a behavior policy $\beta$ such that
\begin{align*}
%================================================================
&{\mathbb P}[a = 1|s = 1] = 0.2,\quad {\mathbb P}[a = 2|s = 1] = 0.8,\\
%================================================================
&{\mathbb P}[a = 1|s = 2] = 0.7,\quad {\mathbb P}[a = 2|s = 2] = 0.3,
%================================================================
\end{align*}
and
\begin{align*}
%================================================================
R_1  = \begin{bmatrix}
   3  \\
   1  \\
\end{bmatrix},\quad R_2  = \begin{bmatrix}
   2  \\
   1  \\
\end{bmatrix}
%================================================================
\end{align*}

Simulated trajectories of the O.D.E. model of Q-learning including the upper and lower comparison systems are depicted in~\cref{fig:appendix:1}.
\begin{figure}[t]
%=============================================================================================================
\centering\includegraphics[width=16cm,height=12cm]{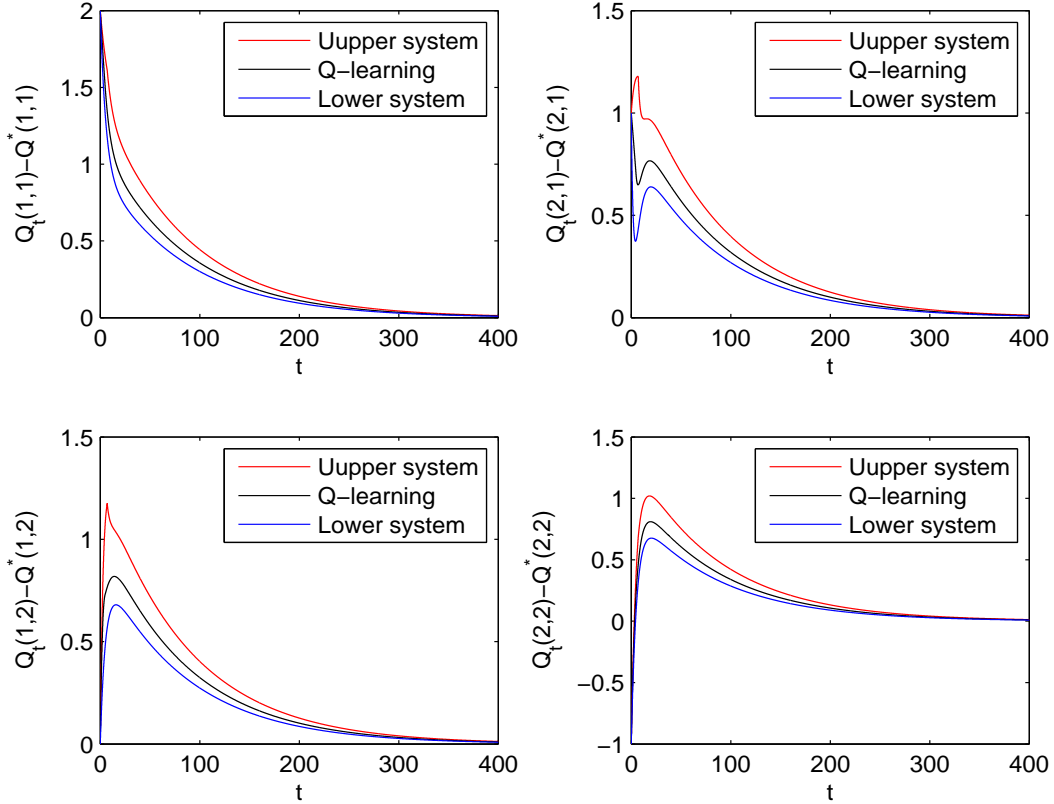}
\caption{Trajectories of the original O.D.E. model of Q-learning and the corresponding upper and lower comparison systems}\label{fig:appendix:1}
%=============================================================================================================
\end{figure}

The simulation study empirically proves that the O.D.E. model associated with the Q-learning is asymptotically stable. Moreover, we compare the trajectory of the O.D.E. model (black line)
with the upper and lower comparison systems (red and blue lines, respectively), and this result shows that the bounding rules that we predicted theoretically hold.

%%%%%%%%%%%%%%%%%%%%%%%%%%%%%%%%%%%%%%%%%%%%%%%%%%%%%%%%%%%%%%%%%%%%%%%%%%%%%%%%
\section{Convergence of Averaging Q-learning}

In this section, we study convergence of the so0called averaging Q-learning. Averaging Q-learning is a newly introduced class of Q-learning motivated by the averaging TD-learning in~\cite{lee2019target}. The averaging TD-learning algorithm is a variation of standard TD-learning~\citep{sutton1988learning} which provides a trade-off between convergence speed and variation of iterations. The averaging Q-learning maintains two separate estimates for the Q-function, the target estimate and the online estimate, respectively, and the target estimate update follows the Polyak's averaging of the current online and target variables to improve the stability. We provide new O.D.E analysis of both algorithms to demonstrate potential utility of the switching system models. Especially, the analysis for averaging Q-learning is new in the literature. We first introduce the averaging Q-learning, that naturally extends the averaging TD-learning in~\cite{lee2019target} to Q-learning. The full algorithm is described in~\cref{algo:averaging-Q-learning}.

\begin{algorithm}[t]
%=========================================================================================
\caption{Averaging Q-Learning}
  \begin{algorithmic}[1]
    \State Initialize $Q_0^A$ and $Q_0^B$ randomly.
    \For{iteration $k=0,1,\ldots$}
    	\State Sample $(s,a)$
        \State Sample $s'\sim P_a(s,\cdot)$ and $r_a(s,s')$
        \State Update $Q^A_{k+1}(s,a)=Q^A_k(s,a)+\alpha_k \{r_a(s,s')+\gamma\max_{a\in {\cal A}} Q^B_k(s',a)-Q^A_k(s,a)\}$
        \State Update $Q^B_{k+1}=Q^B_k+\alpha_k\delta(Q^A_k-Q^B_k)$
    \EndFor
  \end{algorithmic}\label{algo:averaging-Q-learning}
%=========================================================================================
\end{algorithm}

We now analyze the convergence of the averaging Q-learning algorithm based on the same switching system approach. Following similar lines of the standard Q-learning, the corresponding switching system-based ODE model is given by

\begin{align*}
%================================================================
\frac{d}{dt}\begin{bmatrix}
   Q_t^A\\
   Q_t^B\\
\end{bmatrix}=\begin{bmatrix}
   -D & \gamma DP\Pi_{\pi_{Q_t^B}}\\
   \delta I & -\delta I\\
\end{bmatrix} \begin{bmatrix}
   Q_t^A\\
   Q_t^B\\
\end{bmatrix}+\begin{bmatrix}
   DR\\
   0\\
\end{bmatrix},\quad \begin{bmatrix}
   Q_0^A\\
   Q_0^B\\
\end{bmatrix}\in {\mathbb R}^{2|{\cal S}||{\cal A}|},
%================================================================
\end{align*}
or equivalently,
\begin{align}
%================================================================
\frac{d}{dt}\begin{bmatrix}
   Q_t^A-Q^*\\
   Q_t^B-Q^*\\
\end{bmatrix}=& \begin{bmatrix}
   -D & \gamma D P\Pi_{\pi_{Q_t^B}}\\
   \delta I & -\delta I\\
\end{bmatrix} \begin{bmatrix}
   Q_t^A - Q^*\\
   Q_t^B - Q^*\\
\end{bmatrix}+\begin{bmatrix}
   \gamma DP(\Pi_{\pi_{Q_t^B}}-\Pi_{\pi_{Q^*}})Q^*\\
   0 \\
\end{bmatrix},\nonumber\\
%================================================================
&\begin{bmatrix}
   Q_0^A-Q^*\\
   Q_0^B-Q^*\\
\end{bmatrix}=z \in {\mathbb R}^{2|{\cal S}||{\cal A}|}, \label{eq:switching-system-model-for-AQ-learning}
%================================================================
\end{align}
which matches with the switching system form in~\eqref{eq:swithcing-system-form}. We first establish the global asymptotic stability of~\eqref{eq:switching-system-model-for-AQ-learning}.

\begin{theorem}\label{thm:averaging-Q-learning-stability}
%================================================================
For any $\delta>0$, the origin is the unique globally asymptotically stable equilibrium point of the affine switching system~\eqref{eq:switching-system-model-for-AQ-learning}.
%================================================================
\end{theorem}
\begin{proof}
%================================================================
Using $\gamma DP(\Pi_{\pi_{Q_t^B}}-\Pi_{\pi_{Q^*}})Q^*\leq 0$, we obtain
\begin{align*}
%================================================================
\begin{bmatrix} -D & \gamma DP\Pi_{\pi_{Q_t^B}}\\
   \delta I & -\delta I\\
\end{bmatrix}\begin{bmatrix}
   Q_t^A-Q^*\\
   Q_t^B-Q^*\\
\end{bmatrix}&+\begin{bmatrix}
   \gamma DP(\Pi_{\pi_{Q_t^B}}-\Pi_{\pi_{Q^*}})Q^*\\
   0 \\
\end{bmatrix}\\
&\le\begin{bmatrix}
   - D &\gamma DP\Pi_{\pi_{Q_t^B}}\\
   \delta I &-\delta I\\
\end{bmatrix} \begin{bmatrix}
   Q_t^A-Q^*\\
   Q_t^B-Q^*\\
\end{bmatrix}\\
%================================================================
&\le\begin{bmatrix}
   - D &\gamma DP\Pi_{\pi_{Q_t^B-Q^*}}\\
   \delta I &-\delta I\\
\end{bmatrix} \begin{bmatrix}
   Q_t^A-Q^*\\
   Q_t^B-Q^*\\
\end{bmatrix}.
%================================================================
\end{align*}

Consider the upper comparison system
\begin{align*}
%================================================================
&\frac{d}{dt}\begin{bmatrix}
   Q_t^{A,u}-Q^*\\
   Q_t^{B,u}-Q^*\\
\end{bmatrix}=\begin{bmatrix}
   - D & \gamma DP\Pi_{\pi_{Q_t^{B,u}-Q^*}}\\
   \delta I & -\delta I\\
\end{bmatrix} \begin{bmatrix}
   Q_t^{A,u}-Q^*\\
   Q_t^{B,u}-Q^*\\
\end{bmatrix},\;\begin{bmatrix}
   Q_0^{A,u}-Q^*\\
   Q_0^{B,u}-Q^*\\
\end{bmatrix}> \begin{bmatrix}
   Q_0^A-Q^*\\
   Q_0^B-Q^*\\
\end{bmatrix}\in {\mathbb R}^{2|{\cal S}||{\cal A}|},
%================================================================
\end{align*}
and define the vector functions
\begin{align*}
%================================================================
\overline{f}(y_1,y_2):=&\begin{bmatrix}
   \overline{f}_1(y_1,y_2)\\
   \overline{f}_2(y_1,y_2)\\
\end{bmatrix}:=\begin{bmatrix}
   -D & \gamma DP\Pi_{\pi_{y_2}}\\
   \delta I & -\delta I\\
\end{bmatrix} \begin{bmatrix}
   y_1\\
   y_2\\
\end{bmatrix}\\
%================================================================
\underline{f}(y_1,y_2):=&\begin{bmatrix}
   \underline{f}_1(z_1,z_2)\\
   \underline{f}_2(z_1,z_2)\\
\end{bmatrix}:=\begin{bmatrix}
   -D & \gamma DP\Pi_{\pi_{z_2+ Q^*}}\\
   \delta I & -\delta I\\
\end{bmatrix} \begin{bmatrix}
   z_1\\
   z_2\\
\end{bmatrix}+\begin{bmatrix}
   \gamma DP(\Pi_{\pi_{z_2+Q^*}}-\Pi_{\pi_{Q^*}})Q^*\\
   0\\
\end{bmatrix},
%================================================================
\end{align*}
and consider the systems
\begin{align*}
%================================================================
\frac{d}{dt}\begin{bmatrix}
   y_{t,1}\\
   y_{t,2}\\
\end{bmatrix}=\begin{bmatrix}
   \overline{f}_1(y_{t,1},y_{t,2})\\
   \overline{f}_2(y_{t,1},y_{t,2})\\
\end{bmatrix},\quad y_0 > \begin{bmatrix}
   Q_0^A-Q^*\\
   Q_0^B-Q^*\\
\end{bmatrix},\\
%================================================================
\frac{d}{dt}\begin{bmatrix}
   z_{t,1}\\
   z_{t,2}\\
\end{bmatrix}=\begin{bmatrix}
   \underline{f}_1(z_{t,1},z_{t,2})\\
   \underline{f}_2(z_{t,1},z_{t,2})\\
\end{bmatrix},\quad z_0 = \begin{bmatrix}
   Q_0^A-Q^*\\
   Q_0^B-Q^*\\
\end{bmatrix},
%================================================================
\end{align*}
for all $t \geq 0$. We first prove that $\overline{f}$ is quasi-monotone increasing. We will check the condition of the quasi-monotone increasing function for $\overline{f}_1$ and $\overline{f}_2$, separately. Assume that $p_1\in {\mathbb R}^{|{\cal S}|||{\cal A}|}$ and $p_2\in {\mathbb R}^{|{\cal S}|||{\cal A}|}$ are nonnegative vectors, and an $i$the element of $p_1$ is zero. For $\overline{f}_1$, we have
\begin{align*}
%================================================================
e_i^T \overline{f}_1(y_1+p_1,y_2+p_2)=&-e_i^T D(y_1+p_1)+\gamma e_i^T DP\Pi_{\pi_{(y_2+p_2)}}(y_2+p_2)\\
%================================================================
=&-e_i^T Dy_1+\gamma e_i^T DP\Pi_{\pi_{(y_2+p_2)}}(y_2+p_2)\\
%================================================================
\ge& -e_i^T Dy_1+\gamma e_i^T DP\Pi_{\pi_{y_2}} y_2\\
%================================================================
=& e_i^T \overline{f}_1(y_1,y_2),
%================================================================
\end{align*}
where the second line is due to $-e_i^T D p_1 = 0$. Similarly, assuming that $p_1\in {\mathbb R}^{|{\cal S}|||{\cal A}|}$ and $p_2\in {\mathbb R}^{|{\cal S}|||{\cal A}|}$ are nonnegative vectors, and an $i$the element of $p_2$ is zero, we get
\begin{align*}
%================================================================
e_i^T \overline{f}_2(y_1+p_1,y_2+p_2)=&\delta e_i^T(y_1+p_1)-\gamma\delta e_i^T (y_2+p_2)\\
%================================================================
=&\delta e_i^T(y_1+p_1)-\gamma\delta e_i^T y_2\\
%================================================================
\ge&\delta e_i^T y_1-\gamma\delta e_i^T y_2\\
%================================================================
=& e_i^T \overline{f}_2(y_1,y_2),
%================================================================
\end{align*}
where the second line is due to $e_i^T p_2=0$. Therefore, $\overline{f}$ is quasi-monotone increasing. The Lipschitz continuity of $\overline{f}$ and $\underline{f}$ can be easily proved. Therefore, by~\cref{lemma:comparision-principle}, $\begin{bmatrix}
   Q_t^A-Q^*\\
   Q_t^B-Q^*\\
\end{bmatrix}\le\begin{bmatrix}
   Q_t^{A,u}-Q^*\\
   Q_t^{B,u}-Q^*\\
\end{bmatrix}$ holds for all $t\geq 0$, where $\begin{bmatrix}
   Q_t^{A,u}-Q^*\\
   Q_t^{B,u}-Q^*\\
\end{bmatrix}$ is the solution of the upper comparison system.

Moreover, using the inequality $\gamma D P\Pi_{\pi_{Q_t^B}} Q_t^B \ge\gamma DP\Pi_{\pi_{Q^*}} Q_t^B$, we obtain
\begin{align*}
%================================================================
\frac{d}{dt}\begin{bmatrix}
   Q_t^A\\
   Q_t^B\\
\end{bmatrix}=\begin{bmatrix}
   - D & \gamma DP\Pi_{\pi_{Q_t^B}}\\
   \delta I & -\delta I\\
\end{bmatrix}\begin{bmatrix}
   Q_t^A\\
   Q_t^B\\
\end{bmatrix} + \begin{bmatrix}
   DR\\
   0\\
\end{bmatrix}\ge \begin{bmatrix}
   -D & \gamma DP\Pi_{\pi_{Q^*}}\\
   \delta I & -\delta I\\
\end{bmatrix} \begin{bmatrix}
   Q_t^A\\
   Q_t^B\\
\end{bmatrix}+\begin{bmatrix}
   DR\\
   0\\
\end{bmatrix}.
%================================================================
\end{align*}
Using this relation, consider the lower comparison system
\begin{align*}
%================================================================
\frac{d}{dt}\begin{bmatrix}
   Q_t^{A,l}-Q^*\\
   Q_t^{B,l}-Q^*\\
\end{bmatrix}=\begin{bmatrix}
   -D & \gamma DP\Pi_{\pi_{Q^*}}\\
   \delta I & -\delta I\\
\end{bmatrix}\begin{bmatrix}
   Q_t^{A,l}-Q^*\\
   Q_t^{B,l}-Q^*\\
\end{bmatrix},\quad \begin{bmatrix}
   Q_0^{A,l}-Q^*\\
   Q_0^{B,l}-Q^*\\
\end{bmatrix}< \begin{bmatrix}
   Q_0^A-Q^*\\
   Q_0^B-Q^*\\
\end{bmatrix}\in {\mathbb R}^{2|{\cal S}||{\cal A}|},
%================================================================
\end{align*}
or equivalently,
\begin{align*}
%================================================================
\frac{d}{dt} \begin{bmatrix}
   Q_t^{A,l}\\
   Q_t^{B,l}\\
\end{bmatrix}=\begin{bmatrix}
   -D & \gamma DP\Pi_{\pi_{Q^*}}\\
   \delta I & -\delta I\\
\end{bmatrix}\begin{bmatrix}
   Q_t^{A,l}\\
   Q_t^{B,l}\\
\end{bmatrix}+\begin{bmatrix}
   DR\\
   0\\
\end{bmatrix}.
%================================================================
\end{align*}

To proceed, define the vector functions
\begin{align*}
%================================================================
\overline{f}(y_1,y_2):=&\begin{bmatrix}
   \overline{f}_1(y_1,y_2)\\
   \overline{f}_2(y_1,y_2)\\
\end{bmatrix}=\begin{bmatrix}
   \gamma DP\Pi_{\pi_{y_2}} & -D\\
   \delta I & -\delta I\\
\end{bmatrix} \begin{bmatrix}
   y_1\\
   y_2\\
\end{bmatrix}+ \begin{bmatrix}
   DR\\
   0\\
\end{bmatrix}\\
%================================================================
\underline{f}(z_1,z_2):=&\begin{bmatrix}
   \underline{f}_1(z_1,z_2)\\
   \underline{f}_2(z_1,z_2)\\
\end{bmatrix}=\begin{bmatrix}
   -D & \gamma DP\Pi_{\pi_{Q^*}}\\
   \delta I & -\delta I\\
\end{bmatrix} \begin{bmatrix}
   z_1\\
   z_2\\
\end{bmatrix}+\begin{bmatrix}
   DR\\
   0\\
\end{bmatrix},
%================================================================
\end{align*}
and consider the systems
\begin{align*}
%================================================================
\frac{d}{dt}\begin{bmatrix}
   y_{t,1}\\
   y_{t,2}\\
\end{bmatrix}=\begin{bmatrix}
   \overline{f}_1(y_{t,1},y_{t,2})\\
   \overline{f}_2(y_{t,1},y_{t,2})\\
\end{bmatrix},\quad y_0 = \begin{bmatrix}
   Q_0^A-Q^*\\
   Q_0^B-Q^*\\
\end{bmatrix},\\
%================================================================
\frac{d}{dt}\begin{bmatrix}
   z_{t,1}\\
   z_{t,2}\\
\end{bmatrix}=\begin{bmatrix}
   \underline{f}_1(z_{t,1},z_{t,2})\\
   \underline{f}_2(z_{t,1},z_{t,2})\\
\end{bmatrix},\quad z_0 < \begin{bmatrix}
   Q_0^A-Q^*\\
   Q_0^B-Q^*\\
\end{bmatrix},
%================================================================
\end{align*}
for all $t \geq 0$. Similar to the upper comparison systems, we can easily prove that $\overline{f}$ is quasi-monotone increasing, $\overline{f}$ and $\underline{f}$ are Lipschitz continuous. Therefore, applying similar steps as before and using~\cref{lemma:comparision-principle}, we have that $\begin{bmatrix}
   Q_t^A-Q^*\\
   Q_t^B-Q^*\\
\end{bmatrix} \ge \begin{bmatrix}
 Q_t^{A,l}- Q^* \\
 Q_t^{B,l}-Q^* \\
\end{bmatrix}$ holds for all $t\geq 0$, where $\begin{bmatrix}
 Q_t^{A,l}-Q^*\\
 Q_t^{B,l}-Q^*\\
\end{bmatrix}$ is the solution of the linear system

Now, it remains to prove the asymptotic convergence of the comparison systems. For notational convenience, we define $\Pi_\sigma$, $\sigma\in {\cal M}$ as $\Pi_{\pi_{Q_t^B}}$ such that $\sigma=\psi(\pi_{Q_t^B})$. Then, for the upper comparison switching system, we apply~\cref{lemma:fundamental-stability-lemma} with $A_{\sigma}=\begin{bmatrix}
   -D & \gamma D P\Pi_{\sigma}\\
   \delta I & -\delta I\\
\end{bmatrix}$ and $L= \begin{bmatrix}
   I & 0\\
   0 & \gamma^{1/2} I\\
\end{bmatrix}$, which satisfies $LA_\sigma=\bar A_\sigma L$ with $\bar A_\sigma=\begin{bmatrix}
   -D & \gamma^{1/2} D P\Pi_\sigma\\
   \gamma^{1/2}\delta I & -\delta I\\
\end{bmatrix}$. To check the strictly negative row dominating diagonal condition, for $i \in\{1,2,\ldots,|{\cal S}||{\cal A}|\}$, we have
\begin{align*}
%================================================================
[\bar A_\sigma]_{ii}+\sum_{j\in \{1,2,\ldots ,n\} \backslash \{ i\}}{|[\bar A_\sigma]_{ij}|}=& [-D]_{ii}+\gamma^{1/2}[-D ]_{ii} \sum_{j \in \{1,2,\ldots,n\}\backslash \{i\}} {|[P\Pi_\sigma]_{ij}|}\\
%================================================================
\le& [-D]_{ii}+\gamma^{1/2}[-D]_{ii}\\
%================================================================
\le& -1+\gamma^{1/2}
%================================================================
< 0.
%================================================================
\end{align*}

For $i\in \{|{\cal S}||{\cal A}|+1,|{\cal S}||{\cal A}|+2,\ldots,2|{\cal S}||{\cal A}|\}$, we also have
\begin{align*}
%================================================================
[\bar A_\sigma]_{ii}+\sum_{j\in\{1,2,\ldots,n\} \backslash \{i\}} {|[\bar A_\sigma]_{ij}|}=-\delta+\delta\gamma^{1/2}=\delta(-1+\gamma^{1/2})<0
%================================================================
\end{align*}
for any $\delta>0$. Therefore, the strictly negative row dominating diagonal condition is satisfied. By~\cref{lemma:fundamental-stability-lemma}, the origin of the switching system~\eqref{eq:switching-system-model-for-AQ-learning} is globally asymptotically stable. The lower comparison system's stability can be proved in an equivalent way. Since the switching system's solution is upper and lower bounded by the corresponding comparison systems, it asymptotically converges to the origin. This completes the proof.
%================================================================
\end{proof}

As a result, by invoking Borkar and Meyn's theorem and following similar arguments as before, we arrive at
\begin{theorem}\label{thm:averaging-Q-learning-convergence}
%================================================================
Consider~\cref{algo:averaging-Q-learning} and assume the step-sizes satisfy~\eqref{eq:step-size-rule}. Then, for any $\delta>0$, $Q^A_k\to Q^*$ and $Q^B_k\to Q^*$ with probability one.
%================================================================
\end{theorem}

\section{Numerical examples demonstrating~\cref{thm:averaging-Q-learning-stability}}

Consider an MDP with ${\cal S}=\{1,2\}$, ${\cal A}=\{1,2\}$, $\gamma = 0.9$,
\begin{align*}
%================================================================
&P_1=\begin{bmatrix}
   0.2 & 0.8\\
   0.3 & 0.7 \\
\end{bmatrix},\quad P_2=\begin{bmatrix}
   0.5 & 0.5 \\
   0.7 & 0.3\\
\end{bmatrix},
%================================================================
\end{align*}
a behavior policy $\beta$ such that
\begin{align*}
%================================================================
&{\mathbb P}[a = 1|s = 1] = 0.2,\quad {\mathbb P}[a = 2|s = 1] = 0.8,\\
%================================================================
&{\mathbb P}[a = 1|s = 2] = 0.7,\quad {\mathbb P}[a = 2|s = 2] = 0.3,
%================================================================
\end{align*}
and
\begin{align*}
%================================================================
R_1  = \begin{bmatrix}
   3  \\
   1  \\
\end{bmatrix},\quad R_2  = \begin{bmatrix}
   2  \\
   1  \\
\end{bmatrix}
%================================================================
\end{align*}

Simulated trajectories of the O.D.E. model of the averaging Q-learning including the upper and lower comparison systems are depicted in~\cref{fig:appendix:2} for $Q^A_t$ part and~\cref{fig:appendix:3} for $Q^B_t$ part.
\begin{figure}[t]
%=============================================================================================================
\centering\includegraphics[width=16cm,height=12cm]{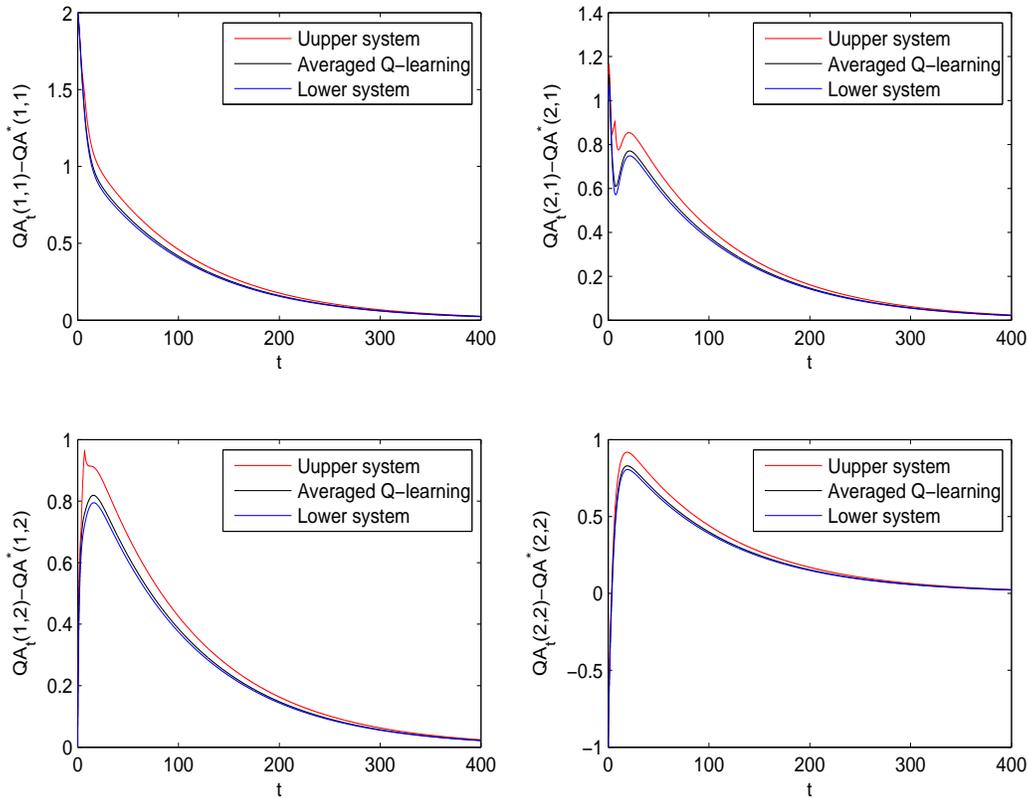}
\caption{Trajectories of the O.D.E. model of averaging Q-learning and the corresponding upper and lower comparison systems ($Q^A_t$ part)}\label{fig:appendix:2}
%=============================================================================================================
\end{figure}
\begin{figure}[t]
%=============================================================================================================
\centering\includegraphics[width=16cm,height=12cm]{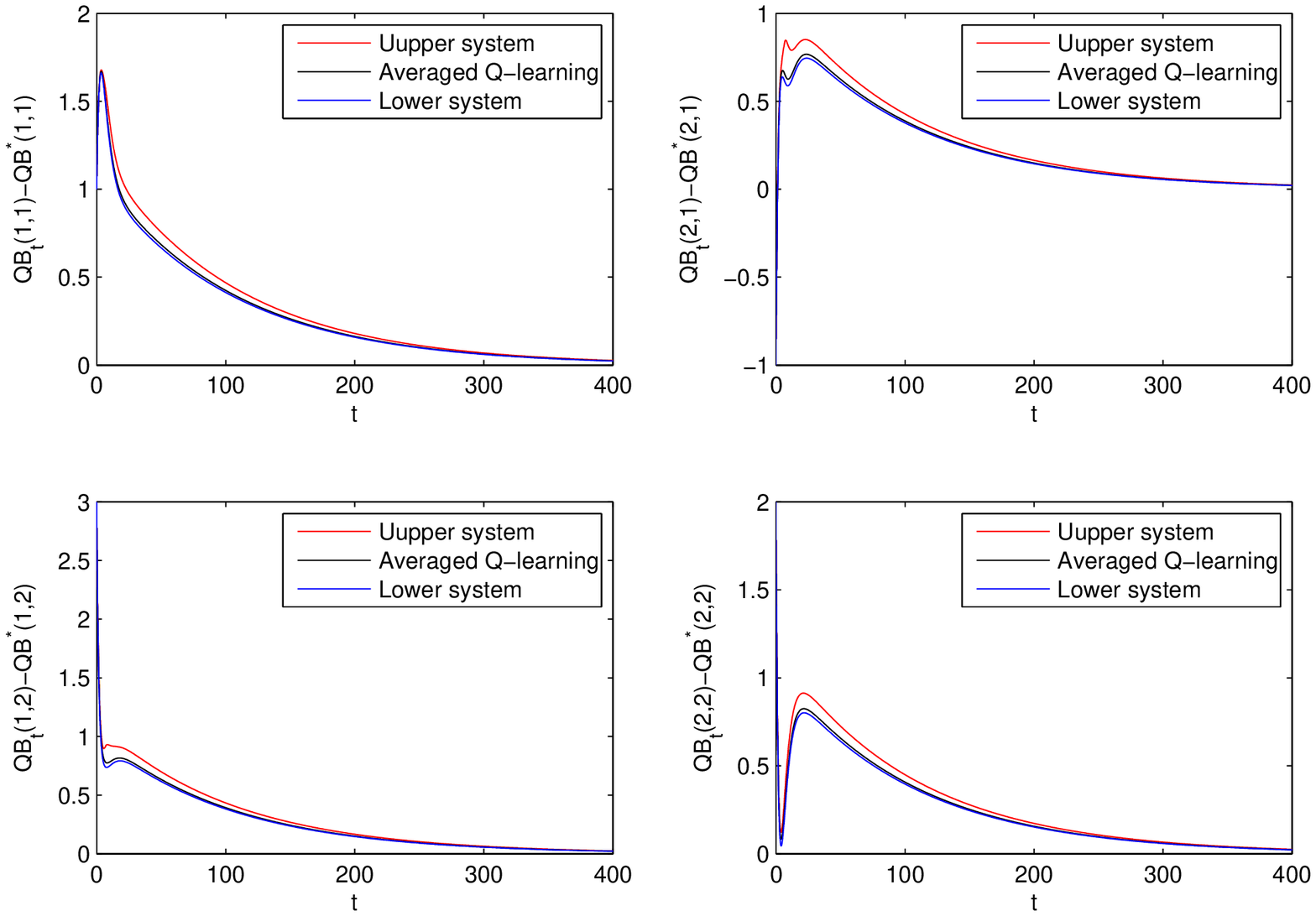}
\caption{Trajectories of the original O.D.E. model of averaging Q-learning and the corresponding upper and lower comparison systems ($Q^B_t$ part)}\label{fig:appendix:3}
%=============================================================================================================
\end{figure}

As before, the simulation study empirically proves that the O.D.E. model associated with the averaging Q-learning is asymptotically stable.

Moreover, the trajectory of the O.D.E. model (black line) is upper and lower bounded by the upper and lower comparison systems, respectively (red and blue lines, respectively).
This result provides numerical evidence that the bounding rules hold.

%%%%%%%%%%%%%%%%%%%%%%%%%%%%%%%%%%%%%%%%%%%%%%%%%%%%%%%%%%%%%%%%%%%%%%%%%%%%%%%%
\section{Convergence of Q-learning with Linear Function Approximation}

When the state-space is large, linear function approximation can be used to approximate the optimal Q-function, $Q^*\cong\Phi\theta^*$, where $\Phi$ is the feature matrix. In particular, given pre-selected basis (or feature) functions $\phi_1,\ldots,\phi_n:{\cal S}\to {\mathbb R}$, the feature matrix $\Phi\in {\mathbb R}^{|{\cal S}| \times n}$ is defined as
\begin{align*}
%================================================================
\Phi:=\begin{bmatrix}
   \phi(1,1)^T\\
   \phi(2,1)^T\\
    \vdots \\
   \phi(|{\cal S}|,|{\cal A}|)^T\\
\end{bmatrix} \in {\mathbb R}^{|{\cal S}||{\cal A}| \times n},
%================================================================
\end{align*}
where $\phi(s,a)^T:=\begin{bmatrix}
   \phi_1(s,a),\phi_2(s,a), \cdots,\phi_n(s,a) \end{bmatrix}\in {\mathbb R}^n$. Here $n \ll |{\cal S}||{\cal A}|$ is a positive integer and $\phi(s)$ is a feature vector.

Q-learning with linear function approximation is described in~\cref{algo:standard-Q-learning-with-linear-function-approximation}, which may not converge in general~\citep{sutton1998reinforcement}. However, under certain conditions, its convergence can be proven. \cite{melo2008analysis} demonstrates the asymptotic convergence when assuming  that the distribution $d_a(s)$ of $(s,a)$ is the stationary distribution under a behavior policy $\beta$, i.e.,
\begin{align}
%================================================================
&d_a(s)=\lim_{k\to\infty} {\mathbb P}[(s_k,a_k)=(s,a)|\beta],\label{eq:stationary-distribution-assumption}
%================================================================
\end{align}
and the following condition holds:
\begin{align}
%================================================================
&\gamma^2\Phi^T\Pi_{\pi}^T D^\beta\Pi_{\pi}\Phi \prec \Phi^T D\Phi,\quad \forall \pi \in \Theta_\Phi,\label{eq:Melo-sufficient-condition}
%================================================================
\end{align}
where $\Theta_\Phi:=\{\pi\in\Theta :\pi(s)=\argmax_{a\in {\cal A}}(\Phi\theta)(s,a),\forall s \in {\cal S}, \theta\in {\mathbb R}^m\}$ and $D^\beta$ is a diagonal matrix whose diagonal entries correspond to the stationary state distribution of the underlying Markov decision process under the behavior policy $\beta$. Recently,~\cite{chen2019performance} considered a slightly stronger condition in order to obtain the convergence rate of Q-learning with linear function approximation.

In this section, we analyze the convergence from the switching system perspective and provide a new sufficient condition that ensures the asymptotic convergence. We start by introducing some basic assumptions that will be needed in the analysis.

\begin{algorithm}[t]
%=========================================================================================
\caption{Standard Q-Learning with linear function approximation}
  \begin{algorithmic}[1]
    \State Initialize $\theta_0 \in {\mathbb R}^n$ randomly.
    \For{iteration $k=0,1,\ldots$}
    	\State Sample $(s,a)$
        \State Sample $s'\sim P_a(s,\cdot)$ and $r_a(s,s')$
        \State Update $\theta_{k+1}=\theta_k +\phi(s,a)\alpha_k\{r_a(s,s')+\gamma\max_{a\in {\cal A}}(\Phi\theta_k)(s',a)-(\Phi\theta_k)(s,a)\}$

    \EndFor

  \end{algorithmic}\label{algo:standard-Q-learning-with-linear-function-approximation}
%=========================================================================================
\end{algorithm}

\begin{assumption}\label{assumption:nonnegative-elements}
%=========================================================================================
$[\Phi]_{ij}\geq 0$ for all $i\in {\cal S}$ and $j\in \{1,2,\ldots,n \}$.
%=========================================================================================
\end{assumption}
\begin{assumption}\label{assumption:orthogonal}
%=========================================================================================
All column vectors of $\Phi$ are orthogonal.
%=========================================================================================
\end{assumption}

\cref{assumption:nonnegative-elements} requires all elements of $\Phi$ to be nonnegative. This assumption is required in our convergence analysis to obtain lower and upper comparison systems of the switched linear system. In the case that no function approximation is used, $\Phi$ is set to be an identity matrix, $\Phi = I$, which automatically satisfies~\cref{assumption:nonnegative-elements}. We emphasize that this assumption is not very restrictive. For instance, if the values of rewards are nonnegative, then it is sufficient to set feature vectors with nonnegative elements when approximating the Q-function. If not, the rewards can always be made nonnegative by adding a large enough constant. \cref{assumption:orthogonal} is slightly stricter than the assumption of full column rank which is usually adopted in the RL literature.

% Therefore, the feature vectors satisfying~\cref{assumption:nonnegative-elements} can be applied to a broad spectrum of MDPs.

Following a similar analysis as the previous section, the associated switched linear system model is given by
\begin{align*}
%================================================================
\frac{d}{dt}\theta_t=(\gamma\Phi^T DP\Pi_{\pi_{\Phi\theta_t}}\Phi-\Phi^T D\Phi)\theta_t+\Phi^T DR,\quad \theta_0\in {\mathbb R}^n,
%================================================================
\end{align*}
or equivalently,
\begin{align}
%================================================================
\frac{d}{dt}(\theta_t-\theta^*)=&(\gamma\Phi^T DP\Pi_{\pi_{\Phi \theta_t}} \Phi-\Phi^T D\Phi)(\theta_t-\theta^*)+\gamma\Phi^T DP(\Pi_{\pi_{\Phi\theta_t}}-\Pi_{\pi_{\Phi\theta^*}})\Phi\theta^*,\nonumber\\
%================================================================
\theta_0-\theta^*=& z\in {\mathbb R}^n,\label{eq:switched-system-model-LFA}
%================================================================
\end{align}
where $\pi_{\Phi\theta_t}(s)=\argmax_{a\in {\cal A}}(\Phi\theta_t)(s,a)$ and $\theta^*$ is the optimal parameter satisfying the projected Bellman equation
\begin{align*}
%================================================================
&\Phi\theta^*=\Gamma(\gamma P\Pi_{\pi_{\Phi\theta^*}}\Phi\theta^*+R),
%================================================================
\end{align*}
and $\Gamma:=\Phi(\Phi^T D\Phi)^{-1}\Phi^T D$ is the projection onto the range of $\Phi$.

We will first establish the asymptotic stability of the system~\eqref{eq:switched-system-model-LFA}.
\begin{proposition}\label{prop:asymptotic-stability-of-switched-system-model-LFA}
%================================================================
Suppose that~\cref{assumption:nonnegative-elements} and~\cref{assumption:orthogonal} hold. The origin is the unique globally asymptotically stable equilibrium point of the affine switching system~\eqref{eq:switched-system-model-LFA} if the following condition holds:
\begin{align*}
%================================================================
&-\phi_i^T D\phi_i+\phi_i^T\gamma DP\Pi_{\psi(\pi)}\sum_{j\in \{1,2,\ldots,n\}}{\phi_j}<0,\quad \pi\in\Theta_\Phi,
%================================================================
% <&0
%================================================================
\end{align*}
where $\Theta_\Phi:=\{\pi\in\Theta :\pi(s)=\argmax_{a\in {\cal A}}(\Phi\theta)(s,a),\forall s \in {\cal S}, \theta\in {\mathbb R}^m\}$ and
\begin{align*}
%================================================================
\phi_i=\begin{bmatrix} \phi_i(1,1) & \phi_i(2,1) & \cdots & \phi_i(|{\cal S}|,|{\cal A}|)\\ \end{bmatrix}^T \in {\mathbb R}^{|{\cal S}||{\cal A}|}.
%================================================================
\end{align*}
For a computationally tractable sufficient condition, $\Theta_\Phi$ can be replaced with $\Theta$.
%================================================================
\end{proposition}
\begin{proof}
%================================================================

By~\cref{assumption:nonnegative-elements}, it holds that $\gamma\Phi^T DP(\Pi_{\pi_{\Phi\theta_t^l}}-\Pi_{\pi_{\Phi\theta^*}})\Phi\theta^*\le 0$, $\Phi^T (\gamma DP\Pi_{\pi_{\Phi\theta_t^u}}-D)\Phi (\theta_t^u-\theta^*) \leq \Phi^T (\gamma DP\Pi_{\pi_{\Phi(\theta_t^u-\theta^*)}}-D)\Phi (\theta_t^u-\theta^*)$, and we obtain the upper comparison system
\begin{align}
%================================================================
\frac{d}{dt}(\theta_t^u - \theta^*)=&\Phi^T (\gamma DP\Pi_{\pi_{\Phi(\theta_t^u-\theta^*)}}-D)\Phi (\theta_t^u-\theta^*),\nonumber\\
%================================================================
\theta_0^u-\theta^*>&\theta_0-\theta^*\in {\mathbb R}^n.\label{eq:upper-comparison-system-LFA}
%================================================================
\end{align}
To proceed, define the vector functions
\begin{align*}
%================================================================
\overline{f}(y)=& \Phi^T (\gamma DP\Pi_{\pi_{\Phi y}}-D)\Phi y\\
%================================================================
\underline{f}(z)=& (\gamma\Phi^T DP\Pi_{\pi_{\Phi (z+\theta^*)}} \Phi-\Phi^T D\Phi)z+\gamma\Phi^T DP(\Pi_{\pi_{\Phi(z+\theta^*)}}-\Pi_{\pi_{\Phi\theta^*}})\Phi\theta^*,,
%================================================================
\end{align*}
and consider the systems
\begin{align*}
%================================================================
\frac{d}{dt}y_t = \overline{f}(y_t),\quad y_0 > \theta_0-\theta^*,\\
%================================================================
\frac{d}{dt}z_t = \underline{f}(z_t),\quad z_0 = \theta_0-\theta^*,
%================================================================
\end{align*}
for all $t \geq 0$. To apply~\cref{lemma:fundamental-stability-lemma}, we first check the quasi-monotonicity of $\overline{f}$. For any nonnegative vector $p$ such that its $i$th element is zero, we have
\begin{align*}
%================================================================
e_i^T \overline{f}(y+p)=& e_i^T(\gamma\Phi^T DP\Pi_{\Phi(y+p)}\Phi-\Phi^T D\Phi)(y+p)\\
%================================================================
=&\gamma e_i^T\Phi^T DP\Pi_{\Phi(y+p)}\Phi (y+p)-e_i^T\Phi^T D\Phi p-e_i^T\Phi^T D\Phi y\\
%================================================================
=&\gamma e_i^T\Phi^T DP\Pi_{\Phi(y+p)}\Phi (y+p)-e_i^T\Phi^T D\Phi y\\
%================================================================
=& \gamma e_i^T\Phi^T DP \begin{bmatrix}
   \max_a(\Phi(y+p))_a(1)\\
   \max_a(\Phi(y+p))_a(2)\\
    \vdots \\
   \max_a(\Phi(y+p))_a(|{\cal S}|)\\
\end{bmatrix}-e_i^T \Phi^T D\Phi y\\
%================================================================
\ge&\gamma e_i^T\Phi^T DP \begin{bmatrix}
   \max_a(\Phi(y))_a(1)\\
   \max_a(\Phi(y))_a(2)\\
    \vdots\\
   \max_a(\Phi(y))_a(|{\cal S}|)\\
\end{bmatrix}-e_i^T\Phi^T D\Phi y\\
%================================================================
=& \gamma e_i^T\Phi^T DP\Pi_{\Phi(y)}\Phi (y)-e_i^T\Phi^T D\Phi y\\
%================================================================
=& e_i^T \overline{f}(y),
%================================================================
\end{align*}
where the third line is due to~\cref{assumption:orthogonal} and the fact that $\Phi^T D\Phi$ is a diagonal matrix. Therefore, $\overline{f}$ is quasi-monotone increasing. The Lipschitz continuity of $\overline{f}$ and $\underline{f}$ can be provided following similar lines of the proof of~\cref{prop:Lipschitz-f}, where we can use the fact that $f(z)=(\gamma DP\Pi_{\pi_{(z+Q^*)}}-D)(z+Q^*)+DR$. Therefore, the vector comparison principle,~\cref{lemma:comparision-principle}, leads to $\theta_t \leq \theta_t^u$ as $t \to \infty$.

On the other hand, by~\cref{assumption:nonnegative-elements}, it holds that $\gamma\Phi^T DP\Pi_{\pi_{\Phi\theta_t}} \Phi\theta_t\ge\gamma\Phi^T DP\Pi_{\pi_{\Phi\theta^*}}\Phi\theta_t$, and we obtain the lower comparison system
\begin{align*}
%================================================================
\frac{d}{dt}(\theta_t^l-\theta^*)=&\Phi^T (\gamma DP\Pi_{\pi_{\Phi\theta^*}}-D)\Phi (\theta_t^l-\theta^*),\\
%================================================================
\theta_0^l-\theta^*<& \theta_0-\theta^*\in {\mathbb R}^n,
%================================================================
\end{align*}
or equivalently,
\begin{align*}
%================================================================
\frac{d}{dt}\theta_t^l=& \Phi^T (\gamma DP\Pi_{\pi_{\Phi\theta^*}}-D)\Phi \theta_t^l - \Phi^T (\gamma DP\Pi_{\pi_{\Phi\theta^*}}-D)\Phi \theta^*,\\
%================================================================
\theta_0^l<& \theta_0\in {\mathbb R}^n.
%================================================================
\end{align*}

To proceed, define the vector functions
\begin{align*}
%================================================================
\overline{f}(y)=&\Phi^T(\gamma DP\Pi_{\pi_{\Phi y}}- D)\Phi y+\Phi^T DR\\
%================================================================
\underline{f}(z)=& \Phi^T (\gamma DP\Pi_{\pi_{\Phi\theta^*}}-D)\Phi z - \Phi^T (\gamma DP\Pi_{\pi_{\Phi\theta^*}}-D)\Phi \theta^*,
%================================================================
\end{align*}
and consider the systems
\begin{align*}
%================================================================
\frac{d}{dt}y_t = \overline{f}(y_t),\quad y_0 = \theta_0,\\
%================================================================
\frac{d}{dt}z_t = \underline{f}(z_t),\quad z_0 < \theta_0,
%================================================================
\end{align*}
for all $t \geq 0$. To apply~\cref{lemma:fundamental-stability-lemma}, we check the quasi-monotonicity of $\overline{f}$, which can be easily proved following the steps for the upper comparison system. The Lipschitz continuity of $\overline{f}$ and $\underline{f}$ can be also proved following similar lines of the proof of~\cref{prop:Lipschitz-f}. Therefore,~\cref{lemma:comparision-principle} leads to $\theta_t \geq \theta_t^l$ as $t \to \infty$.

To prove the asymptotic stability of the original system~\eqref{eq:switched-system-model-LFA}, it is sufficient to prove that the upper and lower comparison systems are globally asymptotically stable. In this respect, we can apply~\cref{lemma:fundamental-stability-lemma} to obtain a sufficient condition for the stability. In particular, both the upper and lower comparison systems are globally asymptotically stable if the switching system is globally asymptotically stable
\begin{align*}
%================================================================
&\frac{d}{dt}\theta_t=A_{\sigma_t}\theta_t,
%================================================================
\end{align*}
under arbitrary switchings, $\sigma_t$, where $A_{\psi(\pi)}=\Phi^T (\gamma DP\Pi_{\psi(\pi)}-D)\Phi$ for all $\pi\in\Theta_\Phi$. By~\cref{lemma:fundamental-stability-lemma}, it is true if and only if
\begin{align*}
%================================================================
&[A_{\psi(\pi)}]_{ii}+\sum_{j\in \{1,2,\ldots,n\} \backslash \{i\}}{|[A_{\psi(\pi)}]_{ij}|}\\
%================================================================
=&[\Phi^T(\gamma DP\Pi_{\psi(\pi)}-D)\Phi]_{ii}+\sum_{j\in \{ 1,2,\ldots,n\}\backslash \{i\}}{|[\Phi^T (\gamma DP\Pi_{\psi(\pi)}-D)\Phi]_{ij}|}\\
%================================================================
=& \phi_i^T (\gamma DP\Pi_{\psi(\pi)})\phi_i-\phi_i^T D\phi_i+\sum_{j\in \{1,2,\ldots,n\} \backslash \{i\}} {\phi_i^T(\gamma DP\Pi_{\psi(\pi)}-D)\phi_j}\\
%================================================================
=&-\phi_i^T D\phi_i+\sum_{j\in \{1,2,\ldots ,n\}}{\phi_i^T\gamma DP\Pi_{\psi(\pi)}\phi_j}\\
%================================================================
=&-\phi_i^T D\phi_i+\phi_i^T \gamma DP\Pi_{\psi(\pi)} \sum_{j\in \{1,2,\ldots,n\}}{\phi_j}\\
%================================================================
<&0
%================================================================
\end{align*}
for all $i\in\{1,2,\ldots,n\},\pi\in\Theta_\Phi$, where the second line is due to~\cref{assumption:nonnegative-elements}, and the fourth line is due to~\cref{assumption:orthogonal} and the fact that $\phi_i^T D\phi_j = 0$ for $j \neq i$. This completes the proof.
%================================================================
\end{proof}

Since the underlying switching system ODE model turns out to be asymptotically stable under the condition in~\cref{prop:asymptotic-stability-of-switched-system-model-LFA}, we can also prove the convergence of~\cref{algo:standard-Q-learning-with-linear-function-approximation}.
\begin{proposition}\label{prop:sufficient-condition-Q-learning-LFA}
%================================================================
Suppose that~\cref{assumption:nonnegative-elements} and~\cref{assumption:orthogonal} hold. The iterates of~\cref{algo:standard-Q-learning-with-linear-function-approximation} converges to $\theta^*$ if the condition in~\cref{prop:asymptotic-stability-of-switched-system-model-LFA} holds.
%================================================================
\end{proposition}
\begin{proof}
%================================================================
The proof can be completed following similar steps used to prove the convergence of Q-learning in~\cref{subsec:proof-of-Q-learning} and using~\cref{lemma:Borkar}.
%================================================================
\end{proof}

\begin{proposition}
%================================================================
Suppose that~\cref{assumption:nonnegative-elements} and~\cref{assumption:orthogonal} hold. In addition, assume that the elements of the feature matrix $\Phi$ are binary numbers, i.e., $\{0,1\}$. Then, the condition in~\cref{prop:sufficient-condition-Q-learning-LFA} always holds.
%================================================================
\end{proposition}
\begin{proof}
%================================================================
If the elements of the feature matrix $\Phi$ are binary numbers, then since the columns of $\Phi$ consist of sums of distinct basis vectors, $e_i \in {\mathbb R}^{|{\cal S}||{\cal A}|}$, and it follows that
\begin{align}
%================================================================
\sum_{j\in \{ 1,2, \ldots ,n\}}{\phi_j}  \leq {\bf 1}_{|{\cal S}||{\cal A}|},\label{eq:3}
%================================================================
\end{align}
where ${\bf 1}_{|{\cal S}||{\cal A}|}$ is the vector with all elements being ones. The right-hand side of the condition in~\cref{prop:sufficient-condition-Q-learning-LFA} is bounded as
\begin{align*}
%================================================================
[A_{\psi(\pi)}]_{ii}+\sum_{j\in \{1,2,\ldots,n\}\backslash \{i\}}{|[A_{\psi(\pi)}]_{ij}|}=&-\phi_i^T D\phi_i+\phi_i^T \gamma DP\Pi_{\psi(\pi)}\sum_{j\in\{1,2,\ldots,n\}}{\phi_j}\\
%================================================================
\le&-\phi_i^T D\phi_i+\phi_i^T \gamma DP\Pi_{\psi(\pi)}{\bf 1}_{|{\cal S}||{\cal A}|}\\
%================================================================
=&-\phi_i^T D\phi_i+\gamma\phi_i^T D{\bf 1}_{|{\cal S}||{\cal A}|}\\
%================================================================
=&-\phi_i^T D\phi_i+\gamma\phi_i^T D\phi_i\\
%================================================================
=& (\gamma-1)\phi_i^T D\phi_i\\
%================================================================
<&0,
%================================================================
\end{align*}
where the first line comes from~\cref{prop:sufficient-condition-Q-learning-LFA}, the second line is due to~\eqref{eq:3}, and the third line is due to the fact that $P\Pi_{\psi(\pi)}$ is a stochastic matrix, i.e., its low sums are one. This completes the proof.
%================================================================
\end{proof}

In the sequel, we give a simple MDP which satisfies the sufficient condition in~\cref{prop:sufficient-condition-Q-learning-LFA}.
\begin{example}
%================================================================
Consider an MDP with ${\cal S}=\{1,2\}$, ${\cal A}=\{1,2\}$, $\gamma=0.9$,
\begin{align*}
%================================================================
&P_1=\begin{bmatrix}
   1/2 & 1/2\\
   1 & 0 \\
\end{bmatrix},\quad P_2=\begin{bmatrix}
   0 & 1 \\
   2/3 & 1/3\\
\end{bmatrix},\\
%================================================================
&D= \begin{bmatrix}
   1/4 & 0 & 0 & 0\\
   0 & 1/4 & 0 & 0\\
   0 & 0 & 1/4 & 0\\
   0 & 0 & 0 & 1/4\\
\end{bmatrix},\quad \Phi = \begin{bmatrix}
   1 & 0\\
   1 & 0\\
   0 & 1\\
   0 & 1\\
\end{bmatrix}.
%================================================================
\end{align*}

The quantities in~\cref{prop:sufficient-condition-Q-learning-LFA} are given by
\begin{align*}
%================================================================
&[A_{\psi(\pi)}]_{11}+|[A_{\psi(\pi)}]_{12}|=-0.1625+0.1125<0,\\
%================================================================
&[A_{\psi(\pi)}]_{22}+|[A_{\psi(\pi)}]_{21}|=-0.2+0.15<0.
%================================================================
\end{align*}
for all possible deterministic policies $\pi\in\Theta$. Therefore,~\cref{algo:standard-Q-learning-with-linear-function-approximation} converges to the optimal $\theta^*$.
%================================================================
\end{example}

% A natural question is how conservative  is the condition in~\cref{prop:sufficient-condition-Q-learning-LFA}. Although it is not a simple question in general, we can at least prove that Melo's sufficient condition~\eqref{eq:Melo-sufficient-condition} implies the condition in~\cref{prop:sufficient-condition-Q-learning-LFA}.
Comparing to the Melo's sufficient condition~\eqref{eq:Melo-sufficient-condition}, the new condition provided in~\cref{prop:sufficient-condition-Q-learning-LFA} applies to arbitrary state-action distributions, not limited to the stationary distributions of behavior policies. Although the assumptions adopted here are somewhat restrictive, we can prove that under these assumptions, the condition in~\cref{prop:sufficient-condition-Q-learning-LFA} is actually less conservative than Melo's sufficient condition~\eqref{eq:Melo-sufficient-condition}.

% Note that these two conditions are based on slightly different assumptions. For instance,~\cref{assumption:nonnegative-elements} is not required in~\cite{melo2008analysis}, while the assumption~\eqref{eq:stationary-distribution-assumption} is not required in~\cref{prop:sufficient-condition-Q-learning-LFA} (it admits arbitrary state-action distributions).
% Therefore, we need to first make assumptions which are common in both conditions.

\begin{proposition}\label{prop:less-conservative}
%================================================================
Assume that~\cref{assumption:nonnegative-elements} and~\cref{assumption:orthogonal} hold, and~\eqref{eq:stationary-distribution-assumption} is satisfied. Then, Melo's sufficient condition~\eqref{eq:Melo-sufficient-condition} implies the condition in~\cref{prop:sufficient-condition-Q-learning-LFA}.
%================================================================
\end{proposition}
\begin{proof}
%================================================================
The basic idea of the proof relies on the fact that Melo's sufficient condition ensures the existence of a quadratic Lyapunov function for the upper comparison system~\eqref{eq:upper-comparison-system-LFA} following the results in~\cite{melo2008analysis}. Since the new sufficient condition,~\cref{prop:less-conservative}, is a necessary and sufficient condition for the global asymptotic stability of the upper comparison system, the Melo's condition implies the proposed new condition. Suppose that Melo's sufficient condition holds, and consider the quadratic Lyapunov function candidate
\begin{align*}
%================================================================
V(\theta_t-\theta^*):=\frac{1}{2}(\theta_t-\theta^*)^T (\theta_t-\theta^*).
%================================================================
\end{align*}

Its time derivative along the state trajectories of the upper comparison system~\eqref{eq:upper-comparison-system-LFA} is given by
\begin{align*}
%================================================================
&\frac{d}{dt}V(\theta_t-\theta^*)=(\theta_t-\theta^*)^T\Phi^T(\gamma DP\Pi_{\pi_{\Phi\theta_t}}-D)\Phi(\theta_t-\theta^*)\\
%================================================================
=&\gamma (\theta_t-\theta^*)^T \Phi^T DP\Pi_{\pi_{\Phi\theta_t}}\Phi(\theta_t-\theta^*)-(\theta_t-\theta^*)^T\Phi^T D\Phi(\theta_t-\theta^*)\\
%================================================================
=&-(\theta_t-\theta^*)^T \Phi^T D\Phi (\theta_t-\theta^*)+\gamma {\mathbb E}[(\theta_t-\theta^*)^T\Phi^T (e_a\otimes e_s)(e_{s'})^T\Pi_{\pi_{\Phi\theta_t}}\Phi(\theta_t-\theta^*)\},
%================================================================
\end{align*}
where $(s,a)$ is sampled from the stationary state-action distribution and $s'\sim P_a(s,\cdot)$. Similar to the ideas in~\cite{melo2008analysis}, using Holder's inequality leads to
\begin{align*}
%================================================================
\frac{d}{dt}V(\theta_t-\theta^*)=&(\theta_t-\theta^*)^T\Phi^T(\gamma DP\Pi_{\pi_{\Phi\theta_t}}-D)\Phi(\theta_t-\theta^*)\\
%================================================================
\le& -(\theta_t-\theta^*)^T\Phi^T D\Phi(\theta_t-\theta^*)\\
%================================================================
&+\gamma\sqrt{{\mathbb E}[(\theta_t-\theta^*)^T\Phi^T(e_a\otimes e_s)(e_a\otimes e_s)^T \Phi(\theta_t-\theta^*)]}\\
%================================================================
&\times \sqrt{ {\mathbb E}[(\theta_t-\theta^*)^T\Phi^T\Pi_{\pi_{\Phi\theta_t}}^T (e_{s'})(e_{s'})^T\Pi_{\pi_{\Phi\theta_t}}\Phi(\theta_t-\theta^*)]}\\
%================================================================
=&-(\theta_t-\theta^*)^T\Phi^T D\Phi(\theta_t-\theta^*)\\
%================================================================
&+\gamma\sqrt{(\theta_t-\theta^*)^T \Phi^T D\Phi(\theta_t-\theta^*)}\\
%================================================================
&\times\sqrt{(\theta_t-\theta^*)^T\Phi^T \Pi_{\pi_{\Phi\theta_t}}^T D^\beta\Pi_{\pi_{\Phi\theta_t}}\Phi(\theta_t-\theta^*)},
%================================================================
\end{align*}
where the last equality uses the fact that the distribution of $s'$ is identical to the distribution of $s$. Now, we apply the Melo's condition to have
\begin{align*}
%================================================================
\frac{d}{dt}V(\theta_t-\theta^*)=&(\theta_t-\theta^*)^T\Phi^T(\gamma DP\Pi_{\pi_{\Phi\theta_t}}-D)\Phi(\theta_t-\theta^*)\\
%================================================================
<&-(\theta_t-\theta^*)^T \Phi^T D\Phi(\theta_t-\theta^*)\\
%================================================================
&+\gamma\sqrt{(\theta_t-\theta^*)^T\Phi^T D\Phi(\theta_t-\theta^*)}\sqrt{\frac{1}{\gamma^2}(\theta_t-\theta^*)^T\Phi^T D\Phi(\theta_t-\theta^*)}\\
%================================================================
=&-(\theta_t-\theta^*)^T\Phi^T D\Phi(\theta_t-\theta^*)+(\theta_t-\theta^*)^T\Phi^T D\Phi(\theta_t-\theta^*)\\
%================================================================
=&0,\quad \forall \theta_t-\theta^*\neq 0.
%================================================================
\end{align*}

This implies that $V$ is a Lyapunov function. By the standard Lyapunov theorem, the origin of the upper comparison system~\eqref{eq:upper-comparison-system-LFA} is globally asymptotically stable. The proof holds even if the upper comparison system is arbitrarily switching. Since the new sufficient condition in~\cref{prop:less-conservative} is a necessary and sufficient condition for the global asymptotic stability of the upper comparison system~\eqref{eq:upper-comparison-system-LFA} under arbitrary switching, this implies that the proposed new condition holds. This completes the proof.
%================================================================
\end{proof}

Under~\cref{assumption:nonnegative-elements} and~\cref{assumption:orthogonal},~\cref{prop:less-conservative} essentially implies that the set of MDPs with linear function approximations whose convergence is verifiable by~\cref{prop:sufficient-condition-Q-learning-LFA} includes those based on the condition~\eqref{eq:Melo-sufficient-condition}.

A naturally question that arises here is whether the inclusion is strict. Below we answer the question by providing a simple counter example which satisfies the sufficient condition in~\cref{prop:sufficient-condition-Q-learning-LFA}, while violating the Melo's sufficient condition~\eqref{eq:Melo-sufficient-condition}.
\begin{example}
%================================================================
Consider an MDP with ${\cal S}=\{1,2\}$, ${\cal A}=\{1,2\}$, $\gamma = 0.9$,
\begin{align*}
%================================================================
&P_1=\begin{bmatrix}
   1/2 & 1/2\\
   1 & 0 \\
\end{bmatrix},\quad P_2=\begin{bmatrix}
   0 & 1 \\
   2/3 & 1/3\\
\end{bmatrix},
%================================================================
\end{align*}
and a behavior policy $\beta$ such that
\begin{align*}
%================================================================
&{\mathbb P}[a = 1|s = 1] = 0.2,\quad {\mathbb P}[a = 2|s = 1] = 0.8,\\
%================================================================
&{\mathbb P}[a = 1|s = 2] = 0.7,\quad {\mathbb P}[a = 2|s = 2] = 0.3.
%================================================================
\end{align*}
The corresponding matrices $D^{\beta}$ and $D$ are given by
\begin{align*}
%================================================================
&D^\beta=\begin{bmatrix}
   0.5 & 0  \\
   0 & 0.5 \\
\end{bmatrix},\quad D= \begin{bmatrix}
   0.1 & 0 & 0 & 0\\
   0 & 0.35 & 0 & 0\\
   0 & 0 & 0.4 & 0\\
   0 & 0 & 0 & 0.15\\
\end{bmatrix}.
%================================================================
\end{align*}
If the feature matrix is
\begin{align*}
%================================================================
\Phi^T =\begin{bmatrix}
   1 & 2 & 0 & 1  \\
\end{bmatrix},
%================================================================
\end{align*}
then the set $\Theta_\Phi$ is given by $\Theta_\Phi = \{\pi_1,\pi_2 \}$, where $\pi_1$ is a deterministic policy such that $\pi_1(1) = 1$, $\pi_1(2) = 1$ and $\pi_2$ is a deterministic policy such that $\pi_2(1) = 2$, $\pi_2(2) = 2$, which is obtained by considering three cases, $\theta>0,\theta=0,\theta<0$, Here, we assume that whenever $\{1,2 \}=\argmax_{a\in {\cal A}}(\Phi\theta)(s,a)$, we select $a=1$ in Q-learning. The quantities in~\cref{prop:sufficient-condition-Q-learning-LFA} are given by $-0.885$ and $-0.03$ for all $\pi\in\Theta_\Phi = \{\pi_1,\pi_2 \}$. Therefore,~\cref{algo:standard-Q-learning-with-linear-function-approximation} converges to the optimal $\theta^*$. However, the quantity $\gamma^2\Phi^T\Pi_{\pi}^T D^\beta\Pi_{\pi}\Phi - \Phi^T D\Phi$ is computed as $-1.2450$ and $0.3750$ for all $\pi\in\Theta_\Phi = \{\pi_1,\pi_2 \}$, respectively. This implies that the condition in~\eqref{eq:Melo-sufficient-condition} fails to verify the convergence.
%================================================================
\end{example}

\section{Conclusion}
In this paper, we studied asymptotic convergence of Q-learning and its variants based on switching system models. The approach offers a more unified and straightforward convergence analysis of various Q-learning algorithms based on connections between Q-learning and switching system theories. In particular, we have analyzed asymptotic convergence of the standard Q-learning, averaging Q-learning, and Q-learning with linear function approximation. In addition, a sufficient condition to ensure convergence of Q-learning with linear function approximation has been developed based on a switching system stability theory. Compared to the approaches in~\cite{jaakkola1994convergence,tsitsiklis1994asynchronous} based on the stochastic approximation and contraction property of the Bellman operator, a disadvantage of the proposed approach is that the analysis is under a restricted assumption, i.e., samples of the state-action pair from a fixed distribution instead of arbitrary visits. However, the assumption is common in the ODE-based analysis in the literature. Moreover, we expect that the switching system approach in this paper provides practical and convenient tools for analysis of Q-learning and its variants, and can initiate new RL algorithm developments. Finally, a potential future topic is convergence rate analysis of the Q-learning algorithms and multi-agent Q-learning based on the switching system models.

%\newpage
%\bibliographystyle{siamplain}
\bibliographystyle{plainnat}
\bibliography{reference}

\end{document}